\documentclass[11pt]{article}
\usepackage{amsmath,amssymb,amsthm,mathrsfs}
\usepackage[margin=1in]{geometry}
\usepackage{hyperref}
\hypersetup{
  colorlinks=true,linkcolor=blue,citecolor=blue,urlcolor=blue,
  pdftitle={A matching lower bound for the sextic large sieve and structured prime moments},
  pdfauthor={Yuanda Xu, Hejian Sang}
}
\newtheorem{theorem}{Theorem}[section]
\newtheorem{proposition}[theorem]{Proposition}
\newtheorem{lemma}[theorem]{Lemma}
\newtheorem{corollary}[theorem]{Corollary}
\theoremstyle{remark}
\newtheorem{remark}[theorem]{Remark}
\numberwithin{equation}{section}
\newcommand{\N}{\mathrm N}
\newcommand{\OO}{\mathcal O}
\newcommand{\eps}{\varepsilon}
\newcommand{\rad}{\operatorname{rad}}
\newcommand{\Rea}{\operatorname{Re}}
\newcommand{\Ima}{\operatorname{Im}}
\newcommand{\II}{\mathcal I}
\newcommand{\RR}{\mathcal R}

\title{A matching lower bound for the sextic large sieve\\
and structured prime moments}
\author{Yuanda Xu \quad Hejian Sang}
\date{October 2026}

\begin{document}
\maketitle

\begin{abstract}
We study the contrast between the large-sieve norm and
the second moment of structured prime coefficients for
sextic residue characters over the Eisenstein integers.
For squarefree row and column ideals of norms at most
$A$ and $B$, we prove a lower bound of order
$(AB)^{-\eps}(A+B+A^{5/6}B^{1/3}+A^{1/3}B^{5/6})$,
matching de Faveri's upper bound throughout the full
parameter range. In particular, the squared full-cutoff
norm at balance is at least $Q^{7/6-\eps}$.
The proof uses the uniform unmarked character moments
of OpenAI's September 2026 manuscript, a Gauss second
moment uniform in moving prime exclusions, and a
centered comparison that transports a squarefree-row
signal between column scales.
With the additional unmarked inverse moment and buffered
zero detector from the same source, we derive a low-height
zero-density estimate and bound the second moment of a
fixed smooth prime sum by $Q^{(37-10\sqrt6)/6+\eps}$.
We also treat bounded convolutions of independently
weighted large primes.
An appendix retains the first version's unbalanced lower
bound and its exponent $17/16$, without using the
additional plain or inverse moment estimates.
\end{abstract}

\begin{center}
\begin{minipage}{0.9\textwidth}

\end{minipage}
\end{center}

\section{Introduction}
\label{sec:introduction}

A large-sieve upper bound controls every coefficient vector.
A matching lower bound identifies vectors for which the
characters fail to behave orthogonally. For higher-order
characters these vectors are related to Gauss sums and
metaplectic residues. A structured vector, however, need
not realize the worst-case norm. This paper treats both
phenomena in the sextic family: an optimal lower bound
for the general norm and a stronger average bound for
fixed smooth prime coefficients.

Let $K=\mathbb Q(\omega)$, where $\omega^2+\omega+1=0$,
and let $\OO=\mathbb Z[\omega]$. Fix a finite set $S_0$
of prime ideals containing those above $6$.
Write $I(S_0)$ for the integral ideals prime to $S_0$,
represented by their primary generators.
For such ideals put
\[
                         J_6(a,b)=\left(\frac ab\right)_6,
\]
with value zero at noncoprime pairs. Define the
\emph{squared} operator norm
\begin{equation}\label{eq:xi}
 \Xi_{6,S_0}(A,B)=
 \sup_{\substack{\lambda\ne0\\
       \operatorname{supp}\lambda\subset
       \{b\in I(S_0):b\ {\rm squarefree},\ \N b\le B\}}}
 \frac{\displaystyle
   \sum_{\substack{a\in I(S_0)\\a\ {\rm squarefree},\ \N a\le A}}
       \left|\sum_b\lambda_bJ_6(a,b)\right|^2}
      {\displaystyle\sum_b|\lambda_b|^2}.
\end{equation}
Every length in this paper is an ideal norm.
Here $A$ is the row cutoff and $B$ is the column cutoff.
The quantity $\Xi_{6,S_0}(A,B)$ measures the largest
amplification of squared coefficient norm by the character
matrix; a lower bound therefore requires a coefficient
vector with a large squared row sum.

\subsection{The main results}

De Faveri's fixed-order theorem \cite[Theorem 1.1]{DF}
gives
\begin{equation}\label{eq:known-upper}
 \Xi_{6,S_0}(A,B)\ll_{\eps,S_0}(AB)^\eps
       \left(A+B+A^{5/6}B^{1/3}+A^{1/3}B^{5/6}\right).
\end{equation}
Our first result is the matching lower bound.

\begin{theorem}\label{thm:optimal}
For every $\eps>0$ and $A,B\ge2$,
\[
 \Xi_{6,S_0}(A,B)\gg_{\eps,S_0}(AB)^{-\eps}
       \left(A+B+A^{5/6}B^{1/3}+A^{1/3}B^{5/6}\right).
\]
In particular, after renaming the arbitrary loss,
\[
                         \Xi_{6,S_0}(Q,Q)
                                  \gg_{\eps,S_0}Q^{7/6-\eps}.
\]
\end{theorem}

This result uses the zero-slot case of
\cite[Lemma 18.1]{OAI}, stated with its uniformity
conditions in Theorem \ref{thm:plain-input}.

For the second result, let $S$ be any fixed finite set
containing the primes above $6$.
Let $\mathscr U_S(Q)$ consist of primary, sixth-power-free
indices $u$ with $Q\le\N u<2Q$ and $(u,S)=1$.
Put
\[
                 \psi_u(n)=(u/n)_6,
\]
with the original zero extension, and let $\psi_u^*$ be
its primitive inducing character.
For $Q\ge3$ these rows are nonprincipal.
This indexed family includes characters of orders properly
dividing six.
Define
\begin{equation}\label{eq:density-def}
 \mathcal N_{6,S}(Q;\sigma,T)
 =\sum_{u\in\mathscr U_S(Q)}
   \#\{\rho:L(\rho,\psi_u^*)=0,\
                  \Rea\rho\ge\sigma,\ |\Ima\rho|\le T\},
\end{equation}
with zeros counted with multiplicity.
This count will control the number of rows on which
zeros near the right edge of the critical strip obstruct
cancellation in prime sums.
For a fixed $V\in C_c^\infty((1,2))$, put
\begin{equation}\label{eq:prime-moment-def}
 \mathcal P_{6,S}(Q;V)=
 \sum_{\substack{u\in\mathscr U_S(Q)\\u\ {\rm squarefree}}}
 \left|\sum_{\substack{\pi\ {\rm primary\ prime}\\(\pi,S)=1}}
                 V(\N\pi/Q)(u/\pi)_6\right|^2.
\end{equation}
Both prime ideals above a split rational prime are included.
Thus $\mathcal P_{6,S}(Q;V)$ measures the average squared
size of one fixed smooth prime sum over the squarefree rows.

\begin{theorem}\label{thm:prime-density}
For every fixed $C>0$, $51/100<\sigma\le7/8$, and $\eps>0$,
uniformly for $1\le T\le(\log Q)^C$,
\[
 \mathcal N_{6,S}(Q;\sigma,T)
       \ll_{\sigma,C,\eps,S}Q^{R(\sigma)+\eps},
       \qquad R(\sigma)=\frac{31-32\sigma}{21-12\sigma}.
\]
The same bound holds for the cumulative family
$1<\N u\le Q$.
Moreover, for every fixed $V\in C_c^\infty((1,2))$,
\[
 \mathcal P_{6,S}(Q;V)\ll_{V,\eps,S}Q^{\kappa_6+\eps},
       \qquad
 \kappa_6=\frac{37-10\sqrt6}{6}
                  =2.084183762028\ldots.
\]
\end{theorem}

For this theorem we also use the unmarked inverse moment
and buffered detector in \cite[Lemma 17.6 and Proposition 8.3]{OAI}.
They are stated precisely in Section \ref{sec:zero-density}.
The algebraic row-count function is the unamplified
specialization of the detected-row calculation in \cite{Sang}.
The passage to the prime sum uses a finite-height rectangle.

The direct large-sieve comparison for
\eqref{eq:prime-moment-def} is $Q^{13/6+\eps}$.
Indeed \eqref{eq:known-upper} bounds the squared operator
norm by $Q^{7/6+\eps}$, and the squared norm of the prime
coefficient vector is $O_V(Q)$. Thus Theorem
\ref{thm:prime-density} saves
$13/6-\kappa_6=0.082482904639\ldots$ in that comparison.
Theorem \ref{thm:optimal} shows that the general norm is
sharp, while the prime estimate gives a stronger bound
for these particular coefficients.

For comparison, the exponents of
\cite[Corollary 1.5]{DF} at fixed or polylogarithmic height are
\[
 \begin{array}{c|c|c}
 \sigma&\text{fixed-order sieve density exponent}&R(\sigma)\\ \hline
 2/3&5/6&29/39\\
 3/4&8/11&7/12\\
 5/6&4/7&13/33
 \end{array}
\]
for the same sixth-power-free parameter ordering.

\subsection{Two proof branches and their common framework}

The lower-bound argument constructs a vector with large
row energy and preserves that energy while imposing
squarefree support and changing the column scale.
The prime-moment argument combines two estimates:
rows with far-right zeros are sparse, and the prime sum
is small on rows whose zeros lie farther to the left.
The two branches use the same character family and
plain moment, with different steps after that common input.

The quadratic large sieve has scale $A+B$, up to small
powers, over the rationals \cite{HBq} and number fields
\cite{GL}. Higher-order families have additional terms:
see Heath-Brown's cubic sieve \cite{HB} and the fixed-order
framework of Blomer--Goldmakher--Louvel \cite{BGL}.
Their lower bounds are connected with Gauss-sum bias,
studied in particular in \cite{HBP,DR}.
De Faveri--Dunn--Hoffstein \cite{DDH} proved unconditional
cubic and quartic lower bounds by Rankin--Selberg methods.
Their discussion of higher orders identifies the difficulty
of retaining squarefree Fourier indices and of controlling
the required level correlations.

For the sextic family, the first version of this paper
obtained the matching lower bound in the range
$B^{16/11+\delta}\le A\le B^{2-\delta}$, and hence the
full-cutoff exponent $17/16$.
That argument constructs a squarefree-column Gauss
vector and then transfers its lower bound to squarefree rows.
The present large-sieve branch retains this arithmetic
framework but changes two analytic steps.

First, after row Poisson summation, the local correlation
is decomposed into a divisor-bounded number of one-variable
character products. A second-moment specialization of the
zero-slot plain product estimate bounds the nonexceptional
part uniformly in the moving prime exclusions.
The exceptional frequencies are handled separately.
Their exact divisor structure also permits short row scales.
The resulting Gram main term is a rank-one quadratic form
in the smooth column profile.

Second, we preserve a small-twist Gauss vector in the
heads of the existing row decomposition.
We then subtract two column profiles with the same
Gram main functional. The main terms cancel exactly,
and the remaining squared norm is smaller than the
squarefree-row signal at the unbalanced starting point.
A sequence of adjacent-scale comparisons transports that
signal to the desired column scale.
This uses a main-component comparison of the kind that
also appears in corrected dispersion arguments
\cite[Section 5]{Patterson}.

The structured-coefficient branch is logically separate
from the lower-bound theorem, but uses the same character
conventions and plain moment.
It combines the plain moment with the source inverse
moment and simultaneous witnesses, derives a bad-row count,
and moves a logarithmic-derivative contour only inside
the permitted buffered rectangle.
Proposition \ref{prop:prime-convolution} extends the
result to bounded rough prime convolutions of the type
used in \cite[Section 4]{Patterson}.

\subsection{Analytic inputs and the first-version argument}

The analytic inputs for the two main results and the
first-version result are as follows.
\begin{description}
\item[Original unbalanced bound.]
Finite reciprocity and Gauss identities, two Poisson
summations, and the known cubic and sextic sieve upper
bounds suffice. Appendix \ref{app:unbalanced} gives this
route independently of Theorem \ref{thm:plain-input}.
\item[Optimal full-cutoff bound.]
Theorem \ref{thm:plain-input} supplies the improved
nonexceptional estimate. The uniform Gauss Gram formula,
squarefree-column projection, and centered scale
comparison then prove Theorem \ref{thm:optimal}.
\item[Zero density and prime moments.]
The plain moment is combined with the inverse moment
and buffered detector. The resulting row count and a
finite-height contour shift prove Theorem
\ref{thm:prime-density}.
\end{description}

We use the cited statements of \cite{OAI} in its
30 September 2026 version, including their uniformity
conditions.

\paragraph{Relation to the first version.}
The first version, entitled \emph{A sharp lower bound
for the sextic large sieve}
(\href{https://arxiv.org/abs/2610.17675v1}{arXiv:2610.17675v1}),
established the restricted-range result.
The present version strengthens the range and adds the
structured-coefficient branch. The original argument is retained in Appendix
\ref{app:unbalanced}, using the common local calculations
and row-tail estimates.

\paragraph{Organization.}
Section \ref{sec:setup} collects the common conventions,
known sieve upper bounds, and plain moment.
The large-sieve branch occupies Sections \ref{sec:gram}--%
\ref{sec:optimal-proof}: the Gram formula, column
projection, and full-range lower bound.
The structured-coefficient branch occupies Sections
\ref{sec:zero-density} and \ref{sec:prime-moments}:
zero density first, then prime moments and convolutions.
Appendix \ref{app:unbalanced} recovers the first-version
theorem using the shared arithmetic but no new moment input.

\section{Common framework and analytic inputs}
\label{sec:setup}

\subsection{Characters, fixed exclusions, and ray classes}

For an element or ideal $a$, put $q_a=\N a$.
For elements this is $|a|^2$.
All ideal sums outside $S$ use the primary generator
congruent to one modulo $3$.
For a literal element $x$, write
$\chi_s(x)=(x/(s))_6$, with its zero extension.
Thus $J_6(a,b)=\chi_b(a_{\rm primary})$.
Gauss numerators remain literal elements.
At the primes in $S$, generators are fixed once and
retained as literal numerators whenever those primes
occur in an auxiliary divisor.

We may enlarge $S_0$ to a fixed finite $S$.
This selects a submatrix and gives
\begin{equation}\label{eq:submatrix}
                       \Xi_{6,S_0}(A,B)\ge\Xi_{6,S}(A,B).
\end{equation}
The enlargement will include a sufficiently large fixed
prime cutoff. It never depends on $A,B$ or on a moving
modulus.

Sextic reciprocity in these conventions is
\begin{equation}\label{eq:reciprocity}
 J_6(a,b)=\mathfrak r(a,b)J_6(b,a),\qquad
 \mathfrak r(a,a)=\chi_a(-1),
\end{equation}
where $\mathfrak r$ is a symmetric sign bicharacter of fixed ray
classes. The first identity is used at coprime pairs;
it extends to the others because both symbols vanish.
The phase $\mathfrak r$ itself is defined on all ray classes.
We use \cite[Lemma 4.4]{OAI}.

Choose a fixed $S$-supported modulus $\mathfrak l_0$
carrying the primary condition, $\mathfrak r$, and every character
$n\mapsto\chi_n(h)$ with $h$ supported on $S$, including
unit numerators. Such a modulus exists: these are
finite-order characters ramified only at $S$, and their
exponents may be reduced modulo six on $S$-unit columns.
Let $\mathcal T$ be the full primary unit residue group
modulo $\mathfrak l_0$, and fix a class $\sigma_{\rm ray}$.
In particular every such $S$-supported numerator character
is constant on a column fiber $[s]=\sigma_{\rm ray}$.
This finite refinement of the ray data is made only for
convenience. It does not change the full matrix in
\eqref{eq:xi}, and its size is independent of all large
parameters and moving exclusions.

\subsection{Known arbitrary-coefficient sieve bounds}
\label{subsec:known-sieves}

Let $\Theta_6(A,B)$ denote the squared norm with
sixth-power-free rows and columns, and let $\Xi_3(A,B)$
denote the cubic squared norm with squarefree rows and
columns. An ideal is sixth-power-free if all its prime
exponents are at most five; the unit ideal is included.
The estimates from \cite[Theorem 1.1]{DF} are
\begin{align}
 \Theta_6(A,B)&\ll_\eps(AB)^\eps
       \{A+B+A^{5/6}B^{1/3}+A^{1/3}B^{5/6}\},
       \label{eq:theta6}\\
 \Xi_3(A,B)&\ll_\eps(AB)^\eps
       \{A+B+(AB)^{2/3}\}.
       \label{eq:xi3}
\end{align}
The cubic squarefree case is also \cite[Theorem 2]{HB}.
These are upper bounds for arbitrary coefficient vectors.
They are used both in the row-tail estimates for the
lower-bound branch and in the repeated-prime correction
for structured convolutions.

\begin{remark}[Compatibility of the character families]
\label{rem:family-compatibility}
The literal primary family agrees with the construction
in \cite[Section 2]{DF} for an admissible choice of
representatives. Choose the defining modulus
$\mathfrak c\subseteq(3)$, supported on $S$ and deep
enough at its primes, and choose the primary generator
$m_E$ of each representative ideal. In the defining
presentation $\mathfrak a=(x)E\mathfrak g^6$, with
$x\equiv1\pmod{\mathfrak c}$, the primary generator is
$xm_E\gamma^6$. The presentation can be chosen so that
$\gamma$ is a unit at the column modulus. Thus the
defining symbol is $J_6(a,b)$, including its zero values
at collisions. The identity $J_6(a,b)^2=(a/b)_3$ also
identifies the cubic subfamily used in \eqref{eq:xi3}.
\end{remark}

\subsection{The unmarked plain product moment}

The plain moment controls products of smooth character
sums averaged over rows. Its dependence on moving
prime exclusions will be used after Poisson summation
in the Gauss Gram calculation. In the zero-density
branch, its fourth-moment specialization bounds the
number of rows with a large plain character sum.

In the next theorem, an annular profile means a smooth
function supported in a fixed compact subinterval of
$(0,\infty)$.

\begin{theorem}[Unmarked plain product moment, \cite{OAI}]
\label{thm:plain-input}
Let $Z\ge2$, with all length parameters in prescribed
fixed-power ranges of $Z$. On nonzero element rows
$q_k\ll K_{\rm row}$, let
\[
                    \psi_k(n)=\tau(n)\chi_n(k).
\]
The twist $\tau$ is fixed within the row sum.
The union of the supports of its displayed moving
residue-symbol factors, before cancellations, has
radical norm at most $R$.
Retain all original zero extensions.
An exactly separated redundant zero may be represented
by a fixed squarefree puncture of norm at most $Z^{C_0}$,
for a prescribed $C_0$; this puncture is independent of
the varying row and is the same in both factors.
For
\[
 S_{\psi_k}(N_j;V_j)
       =N_j^{-1/2}\sum_n\psi_k(n)V_j(q_n/N_j),
\]
including that common puncture, one has
\begin{equation}\label{eq:plain-product}
 \sum_{\substack{0<q_k\ll K_{\rm row}\\\psi_k^*\ne1}}
       |S_{\psi_k}(N_1;V_1)S_{\psi_k}(N_2;V_2)|^2
                            \ll_\eps K_{\rm row}RZ^\eps .
\end{equation}
The character $\psi_k^*$ is the primitive inducing
character. There is no zero-slot restriction on the
bounded nonnegative logarithmic lengths of $N_1,N_2$,
and no lower bound on its conductor.
The constants depend on finitely many profile seminorms,
with a fixed polynomial cost in separated norm-twist
heights. They are uniform over the permitted moving
supports, masks, and frozen labels.
\end{theorem}

This is exactly the $z=0$ case of
\cite[Lemma 18.1, equation (18.2)]{OAI}, with
$K_{\rm row}=Z^m$, $R=Z^r$, and effective width $m+r$.
The fixed finite ray family contains the characters of
$\mathcal T$ and the supplementary characters.

\begin{corollary}\label{cor:plain-second}
Under the same conventions,
\begin{equation}\label{eq:plain-second}
 \sum_{\substack{0<q_k\ll K_{\rm row}\\\psi_k^*\ne1}}
       \left|\sum_n\psi_k(n)V(q_n/N)\right|^2
                           \ll_\eps NK_{\rm row}RZ^\eps .
\end{equation}
\end{corollary}

\begin{proof}
Take $N_2=1$ and a fixed profile supported sufficiently
close to one, equal to one there, so that only the unit
ideal occurs. The second factor is $\psi_k(1)=1$,
including every puncture. Multiply \eqref{eq:plain-product}
by $N$.
If a fixed ray-fiber restriction occurs, expand it into
characters before this specialization.
Two separate applications of \eqref{eq:plain-second}
may have different fixed masks; within each application
its mask is also present in its unit-ideal factor.
\end{proof}

\section{A uniform Gauss Gram formula}
\label{sec:gram}

We construct a Gauss-weighted test sum and evaluate its
weighted squared row norm. The resulting main term
supplies the energy for the lower-bound argument;
its dependence on the column profile will also make
comparison between column scales possible.

Use
\[
 e(z)=\exp(4\pi i\,\Ima z/\sqrt3),\qquad
 d\mu(z)=(2/\sqrt3)\,d\Rea z\,d\Ima z,\qquad
 \widehat w(y)=\int_{\mathbb C}w(z)e(-zy)\,d\mu(z).
\]
The lattice $\OO$ is self-dual.
For a primary modulus $s$ outside $S$, put
\[
 G(s,k)=q_s^{-1/2}\sum_{x\bmod s}\chi_s(x)e(kx/s),
                         \qquad G(1,k)=1 .
\]
The row profile $w$ is fixed, radial, nonnegative,
nonzero, smooth, and supported away from zero with
$|z|^2\le1$; also $w\le1$.

Let $\mathfrak d$ be squarefree outside $S$,
$D=\N\mathfrak d$, and
$d_{\mathbf e}=\prod_{p\mid\mathfrak d}p^{e_p}$,
where $e_p\in\{0,\ldots,5\}$.
The pair $(\mathfrak d,\mathbf e)$ records the excluded
primes and residue-symbol twists introduced by arithmetic
decompositions later in the proof.
For a smooth column profile $U$ supported in $(1/2,1)$,
define
\begin{equation}\label{eq:gauss-polynomial}
 F_{\mathfrak d,\mathbf e;U}(m;L)
 =\frac1L
 \sum_{\substack{s\ {\rm primary},\ [s]=\sigma_{\rm ray}\\
                                  (s,S\mathfrak d)=1}}
 U(q_s/L)(q_s/L)^{-1/2}J_6(d_{\mathbf e},s)G(s,-m).
\end{equation}
A superscript ${\rm sf}$ will restrict the columns to
squarefree $s$.
The variable $m$ indexes rows, $L$ is the column norm
scale, and $U$ selects a smooth block of columns.
Use the weighted $\ell^2$ norm
\[
 \|F\|_q^2=
 \sum_{\substack{m\in\OO\setminus\{0\}\\(m,S)=1}}
                         w(m/\sqrt q)|F(m)|^2.
\]
Put
\[
 X=L^2/q,\quad
 \II(U)=\int_0^\infty U(t)t^{-5/6}\,dt,\quad
 \Omega(\mathfrak d,\mathbf e)
  =\prod_{p\mid\mathfrak d}
       \frac{q_p^{-j_p/6}}{(1+q_p^{-1})^2},
 \quad j_p=(-e_p\bmod6).
\]
Here $X$ is the dual norm scale after row Poisson
summation. The functional $\II(U)$ determines the
profile dependence of the Gram main term, while
$\Omega(\mathfrak d,\mathbf e)$ records its local
arithmetic correction. Profiles with the same value
of $\II$ will yield the centered comparison in
Section \ref{sec:optimal-proof}.

\begin{proposition}\label{prop:gram}
For $q,L\ge1$, $X\ge1$, and nonnegative $U$,
\begin{equation}\label{eq:gram}
 \begin{split}
 \|F_{\mathfrak d,\mathbf e;U}(\,\cdot\,;L)\|_q^2
 ={}&C_{S,w}\II(U)^2q^{5/6}L^{-2/3}
                                  \Omega(\mathfrak d,\mathbf e)\\
 &+O_\eps\left((2qLD)^\eps
                     \{(q/L)(1+X^{1/12})+D\}\right),
 \end{split}
\end{equation}
where $C_{S,w}>0$ is independent of
$U,\sigma_{\rm ray},\mathfrak d,\mathbf e$.
The error has finite-seminorm dependence on $U$ and is
uniform on prescribed bounded smooth profile families.
For $X\le1$, the uniform bound is $O(q/L)$.
\end{proposition}

The short-row range $q<L$ and the $D$ term are important
for the scale comparison below. The first version's
two-Poisson argument instead gives a $DX$ error and
requires $q/L$ sufficiently large; that estimate is
recovered in Lemma \ref{lem:classical-gram}.
The only use of the plain moment in the present Gram
proof is the nonexceptional estimate of Lemma
\ref{lem:nonexceptional}. The local identities, complete
means, and lattice counts are common to both routes.

\subsection{The first Poisson summation}

Temporarily sum over all element rows, omitting $(m,S)=1$,
and denote the resulting Gram by $\mathcal H(q,L)$.
The omitted zero row has weight zero.
Write $H_U(t)=U(t)t^{-1}$ and
\[
 a(s)=1_{(s,S\mathfrak d)=1,\ [s]=\sigma_{\rm ray}}J_6(d_{\mathbf e},s).
\]
Removing the square-root normalization of $G$ makes
the column coefficient $L^{-3/2}H_U(q_s/L)$.
Opening the finite Gauss sums and applying row Poisson gives
\begin{equation}\label{eq:first-poisson}
 \mathcal H(q,L)=\frac q{L^3}
 \sum_{u,v}a(u)\overline{a(v)}H_U(q_u/L)H_U(q_v/L)
 \sum_{j\in\OO}\mathcal C(u,v;j)
                       \widehat w(\sqrt q\,j/(uv)),
\end{equation}
where
\[
 \mathcal C(u,v;j)=
 \sum_{\substack{x\bmod u,\ y\bmod v\\vx-uy\equiv j\ (uv)}}
                             \chi_u(x)\overline{\chi_v(y)}.
\]
At $j=0$, the unit conditions imply $u\mid v$ and $v\mid u$.
Thus $u=v$, the correlation is $\varphi(u)$, and the
contribution is $O(q/L)$.

For nonzero frequencies, let $C=(u,v)$, $u=Cn_1$, $v=Cn_2$.
The correlation vanishes unless $j=Ck$.
The finite correlation identity of
\cite[Lemma 13.3]{OAI}, with its local factors evaluated
below, is
\begin{equation}\label{eq:finite-correlation}
 \mathcal C(Cn_1,Cn_2;Ck)
 =\chi_{n_1}(k)\overline{\chi_{n_2}(-k)}
                            \mathfrak r(n_1,n_2)L_C(n_1,n_2;k).
\end{equation}
For $p^c\parallel C$, $P=q_p$, the unit--unit local factor is
\begin{equation}\label{eq:local-correlation}
 P^{c-1}\chi_p(n_1)^c\overline{\chi_p(n_2)^c}
 \begin{cases}
 P-1,&p\mid k,\\
 -1,&p\nmid k,\ 6\nmid c,\\
 P-2,&p\nmid k,\ 6\mid c.
 \end{cases}
\end{equation}
If $p$ divides exactly one column, its value is
$P^{c-1}(P-1)1_{6\mid c}1_{p\nmid k}$.
To see these formulas, reduce the defining congruence
$n_2x-n_1y\equiv k\pmod{p^c}$ modulo $p$.
There are $P^{c-1}$ lifts.
For units and $p\nmid k$, the remaining character sum
is the sum over the nonzero field elements other than one;
for $p\mid k$, the two units are proportional.
The one-nonunit branch has a nonzero complete character
sum only for $6\mid c$.

Extend the product of these formulas to all residual
columns by assigning zero when a prime of $C$ divides both.
This extension agrees with the original correlation on
$(n_1,n_2)=1$.
Insert the exact projector
\[
 1_{(n_1,n_2)=1}
       =\sum_{\mathfrak b\mid n_1,n_2}\mu(\mathfrak b),
                         \qquad n_i=\mathfrak b m_i .
\]
Nonzero terms have $(\mathfrak b,SCk\mathfrak d)=1$.
No restriction $(m_i,\mathfrak b)=1$ is inserted.
The residual columns have the same fixed ray class, so
$\mathfrak r(n_1,n_2)\overline{\chi_{n_2}(-1)}=1$.
All common $C,\mathfrak b$ character factors cancel with
their conjugates.

\subsection{Tensor branches and the nonexceptional estimate}

Fix $g=(\rad C,k)$ and write the literal element $k=gk'$.
At each prime of $C$, choose the unit--unit branch or
one of the two permitted one-nonunit branches.
There are at most $3^{\omega(C)}$ choices.
If $f_i$ is the squarefree product of primes assigned
to the nonunit branch in column $i$, substitute
$m_i=f_i z_i$. These primes have $6\mid c$ and do not
divide $k$. No further unit condition at such a prime
is imposed on $z_i$ without an original reason.

A fixed factor $\chi_p(n)^c$ with $6\mid c$ is exactly
the indicator $1_{p\nmid n}$.
Transfer it to its fixed branch mask, retaining all
surviving phases. The other character factors, after
reciprocity and ray-fiber expansion, have the form
\begin{equation}\label{eq:branch-character}
 \nu_i(n)\chi_n(k')\chi_n(gd_{\mathbf e}C^\flat)
                         1_{(n,P_i)=1},
 \qquad
 C^\flat=\prod_{\substack{p^c\parallel C\\6\nmid c}}p^{c\bmod6}.
\end{equation}
The $\nu_i$ belong to one fixed ray family.
For this frozen branch the punctures $P_i$ are independent
of $k'$. Their norms are bounded by a fixed power of the
ambient parameter. The substitutions also give row scalars
of modulus at most one, not new inner arithmetic coefficients.

Define
\[
 R_{C,g}=\N\rad\!\left(g
              \prod_{\substack{p^c\parallel C\\6\nmid c}}p\right).
\]
The moving support in \eqref{eq:branch-character} is at most
$D R_{C,g}$. If $w_{C,g}$ denotes the absolute local
correlation scalar of the branch, then
\begin{equation}\label{eq:balance}
                         w_{C,g}R_{C,g}/\N g\le q_C .
\end{equation}
Indeed the local upper bounds are
\[
 \begin{array}{c|c|c|c}
 \text{case}&w_p&(R_{C,g})_p&(\N g)_p\\ \hline
 6\nmid c,\ p\nmid k&P^{c-1}&P&1\\
 6\nmid c,\ p\mid k&P^c&P&P\\
 6\mid c,\ p\nmid k&P^c&1&1\\
 6\mid c,\ p\mid k&P^c&P&P
 \end{array}
\]
and the nonunit branches occur only in the third line.

Call a frequency exceptional when
\begin{equation}\label{eq:exceptional}
 6\mid v_p(Ck)\quad(p\notin S\mathfrak d),\qquad
 v_p(k)+e_p\equiv0\pmod6\quad(p\mid\mathfrak d).
\end{equation}
On the refined ray fibers all remaining $S$-supported
characters are constant. Thus \eqref{eq:exceptional}
is exactly the condition for a nonzero complete column mean.
For every other frequency, a good prime contributes a
nontrivial local character to every term of the fixed
ray expansion in \eqref{eq:branch-character}.
Its inducing character is therefore nonprincipal.

\begin{lemma}\label{lem:nonexceptional}
The full nonexceptional contribution to
\eqref{eq:first-poisson} is
$O_\eps(D(2qLD)^\eps)$.
\end{lemma}

\begin{proof}
Put $Z=2qLD$. Use the ball $q_Cq_k\ll X$ and outer
dyadic shells $q_Cq_k\asymp\mathcal R X$, $\mathcal R\ge2$;
write $\mathcal R=1$ for the ball.
The smooth two-column profile is
\[
 H_U(t_1)H_U(t_2)
       \widehat w_{\rm rad}(q_Cq_k/(Xt_1t_2)).
\]
Every prescribed collection of column derivatives has
arbitrary decay in $\mathcal R$ on the outer shells.
Fourier-expand in the two logarithmic column variables,
using fixed annular cutoffs for the inner polynomials.
The coefficient for each varying $k$ has a common bound
\[
 C_{J,B}\mathcal R^{-B}
                    (1+|\xi_1|+|\xi_2|)^{-J}.
\]
Take this majorant outside the row sum, then use
Cauchy--Schwarz in the row.
The height cost in \eqref{eq:plain-second} is integrable
after $J$ is chosen sufficiently large.
These extra derivatives concern the external coupling
profile.

For fixed $C,\mathfrak b,g$ and branch, put
$N=L/(q_Cq_{\mathfrak b})$ and
$K_{\rm row}\asymp\mathcal R X/q_C$.
The two remaining column scales are $N_i=N/\N f_i$,
and the $k'$ row length is $O(K_{\rm row}/\N g)$.
Two separate applications of \eqref{eq:plain-second}
give a bound
\[
 w_{C,g}\sqrt{N_1N_2}\frac{K_{\rm row}}{\N g}D R_{C,g}Z^\eps
                                  \le N\mathcal R XDZ^\eps,
\]
by \eqref{eq:balance}.
Every ray term is nonprincipal here.
Its own fixed mask is used with the unit-ideal factor;
the two applications need not have the same mask.
Row restrictions such as $(k,\mathfrak b)=1$ are dropped
only in the resulting nonnegative moments.
A nonempty subunit row or column scale is comparable
to one and can be rescaled with fixed constants.

On $q_C\asymp C_0$, $q_{\mathfrak b}\asymp B_0$ there
are $O(C_0B_0)$ pairs of frozen labels.
Including $q/L^3$, their bound is
\[
                  \frac q{L^3}C_0B_0
                  \frac L{C_0B_0}\mathcal R XD
                                  =\mathcal R D .
\]
Branch counts, $g$ choices, and dyadic pairs cost a
small power; shell decay makes the sum convergent.

First retain $\mathcal R\le Z^2$, so every scale and mask
is in a fixed-power range of $Z$. This remains true for
$q<L$, since $X\le L^2$.
Before the gcd projector, the omitted shell has
$O(L^2)$ column pairs, correlation bounded by $q_C$,
and $O(\mathcal R X/q_C)$ frequencies.
Its absolute contribution is
$O_B(L\mathcal R^{1-B})$.
Arbitrary Fourier decay makes the discarded tail negligible.
\end{proof}

\subsection{The exceptional column-mean error}
\label{subsec:exceptional-error}

Fix an exceptional frequency before applying the next
estimate. The natural zeros at its primes may enter this
pointwise count.

On its ray fiber, each one-variable branch is a constant
phase times a coprimality mask, possibly after $m_i=f_i z_i$.
For a squarefree puncture $P$ outside $S$, M\"obius
inversion and elementary planar lattice counting give
\begin{equation}\label{eq:punctured-count}
 \sum_{\substack{n\equiv\rho\ (\mathfrak l_0)\\(n,P)=1}}
                         V(q_n/N)
 =\text{complete-density integral}
                  +O_V(N^{1/2}d(P)).
\end{equation}
For a divisor $a\mid P$, division by $a$ leaves a
fixed-modulus ray fiber at scale $N/q_a$.
Its error is $O((N/q_a)^{1/2}+1)$.
For subunit scales the sum is empty and its area term
has the required bound. This proves \eqref{eq:punctured-count}
uniformly in $P$.

For two column scales $N_1,N_2$, the error after separation
is bounded by
\[
 (N_1N_2^{1/2}+N_1^{1/2}N_2+N_1^{1/2}N_2^{1/2})Z^\eps.
\]
On the original column support, $N_i\le N$ and $N\ge1$,
where $N=L/(q_Cq_{\mathfrak b})$.
The error of one branch is therefore
\[
                   O_B(\mathcal R^{-B}w_{C,g}N^{3/2}Z^\eps).
\]
Using $w_{C,g}\le q_C$ and multiplying by $q/L^3$
gives
\begin{equation}\label{eq:exception-tuple-error}
 O_B\left(\mathcal R^{-B}Z^\eps
           qL^{-3/2}q_C^{-1/2}q_{\mathfrak b}^{-3/2}\right).
\end{equation}
The sum over $\mathfrak b$ is convergent.

Fix generators at the primes of $S$.
The exceptional condition expresses the literal product as
\begin{equation}\label{eq:sixth-root}
 Ck=\epsilon e_S(\mathbf r)d_{\mathbf j}a^6,\qquad
 d_{\mathbf j}=\prod_{p\mid\mathfrak d}p^{j_p},\quad
 e_S(\mathbf r)=\prod_{p\in S}\pi_p^{r_p},
\end{equation}
where $\epsilon\in\OO^\times$, $0\le r_p\le5$,
and $a$ represents an integral ideal, including possible
factors at $S$. All six units and these fixed $S$-classes
are allowed by our ray refinement.
Moreover $C\mid a^6$ and $(C,S\mathfrak d)=1$.
On the current shell, $\N a\ll(\mathcal R X)^{1/6}$ and
\[
                     \sum_{C\mid a^6}q_C^{-1/2}
                                 \le d(a^6)\ll_\eta(\N a)^\eta .
\]
Summing \eqref{eq:exception-tuple-error} therefore gives
\begin{equation}\label{eq:exception-error}
                 O_\eps(Z^\eps qL^{-3/2}X^{1/6})
                                      \ll Z^\eps q/L ,
\end{equation}
because $L^{-1/2}X^{1/6}=(qL)^{-1/6}\le1$.

\subsection{Complete means and their sum}
\label{subsec:complete-means}

We include the mean calculation to identify both the moving
local factors and the profile dependence of the main term.
Let $c=v_p(C)$, $j=v_p(k)$, and $P=q_p$.
At $p\notin S\mathfrak d$, the normalized complete
two-column mean of the \emph{full} coefficient is
\[
 \mathfrak m_p(c,j)=
 \begin{cases}
 P^c(1-P^{-1})^2,&c>0,\ j=0,\ 6\mid c,\\
 P^c(1-P^{-1})^3,&c>0,\ j>0,\ 6\mid(c+j),\\
 0,&c>0\text{ otherwise}.
 \end{cases}
\]
For $c=0$ it is one if $j=0$, and
$(1-P^{-1})^2 1_{6\mid j}$ if $j>0$.
These follow by averaging \eqref{eq:local-correlation}
with the numerator characters. For $j=0$, $6\mid c$,
the unit--unit and one-nonunit counts add as
$(P-2)+2=P$; for $j>0$ the numerator characters
force both columns to be units.
At $p\mid\mathfrak d$, the mean is
$(1-P^{-1})^2 1_{j+e_p\equiv0\ (6)}$.

On the fixed ray fibers the $S$-part of the character
is constant and cancels against its conjugate.
Write $\delta_0=1/\N\mathfrak l_0$ for the lattice density
of one of these fixed residue classes.
The fixed-place mean is $\delta_0^2$.

Let $w(z)=w_{\rm rad}(|z|^2)$ and define
\[
 \mathcal J_U(x)=
 \left(\frac{2\pi}{\sqrt3}\right)^2
 \int_0^\infty\!\!\int_0^\infty
     U(t_1)U(t_2)(t_1t_2)^{-1}
            \widehat w_{\rm rad}(x/(t_1t_2))\,dt_1dt_2.
\]
After scaling the columns, the outside factor of each
complete mean is
$(q/L)\mu(\mathfrak b)q_C^{-2}q_{\mathfrak b}^{-2}$.
Before completing its divisor ranges, retain the exact
condition \eqref{eq:sixth-root}.
Since the complete correlation mean is at most $q_C$,
the added tail $q_Cq_{\mathfrak b}\gg L$ has, for each $a$,
the bound
\[
 \sum_{C\mid a^6}q_C^{-1}
             \sum_{q_{\mathfrak b}\gg L/q_C}q_{\mathfrak b}^{-2}
 \ll\sum_{C\mid a^6}q_C^{-1}\min\{1,q_C/L\}
 \ll d(a^6)/L.
\]
Thus the full completion costs
\begin{equation}\label{eq:mean-completion}
             O_\eps(Z^\eps(q/L^2)X^{1/6})\ll Z^\eps q/L.
\end{equation}
This bound, unlike a crude completion using $X/L\ll1$,
also covers short row scales.

For a good prime with total exponent $6j>0$ in $Ck$,
the exact sum over its $C$-valuation is
\[
 (1-P^{-1})^2
 \left(1+P^{-6j}+(1-P^{-1})\sum_{c=1}^{6j-1}P^{-c}\right)
                           =(1-P^{-1})(1-P^{-2}).
\]
The coprime-projector sum has baseline factor $1-P^{-2}$.
Consequently, writing
\[
 \eta_{\mathfrak d}=\prod_{p\mid\mathfrak d}(1-P^{-1})^2,\quad
 A_{\mathbf j}=\N d_{\mathbf j},\quad
 g_{S\mathfrak d}(a)=
       \prod_{\substack{p\mid a\\p\notin S\mathfrak d}}(1-P^{-1}),
\]
the completed mean is
\begin{equation}\label{eq:completed-mean}
 \frac qL\frac{\delta_0^2\eta_{\mathfrak d}}
                  {\zeta_K^{S\mathfrak d}(2)}
 \sum_{\epsilon,\mathbf r}\sum_a g_{S\mathfrak d}(a)
 \mathcal J_U\left(
     \frac{A_{\mathbf j}\N e_S(\mathbf r)(\N a)^6}{X}\right).
\end{equation}
Here $\zeta_K^{S\mathfrak d}$ deletes the indicated Euler
factors. The ideal $a$ is unrestricted at the primes
of $\mathfrak d$ as well as at $S$.

Let $r_K=\operatorname*{Res}_{s=1}\zeta_K(s)$.
Uniformly in $\mathfrak d$ and $Y>0$,
\begin{equation}\label{eq:weighted-ideal-count}
 \sum_{\N a\le Y}g_{S\mathfrak d}(a)
                =\frac{r_K}{\zeta_K^{S\mathfrak d}(2)}Y
                                      +O_K(Y^{1/2}).
\end{equation}
Indeed expand $g(a)$ as
$\sum_{b\mid a,(b,S\mathfrak d)=1}\mu(b)/\N b$ and use
the elementary ideal count $r_KY+O(Y^{1/2})$.
The errors are bounded by
$Y^{1/2}\sum_b(\N b)^{-3/2}$ independently of
the exclusions; completing the main sum costs $O(1)$.
For $Y<1$ the assertion follows directly.

Partial summation in \eqref{eq:completed-mean} now gives
an error $O((q/L)X^{1/12})$.
The moving main factors simplify by
\[
        \frac{(1-P^{-1})^2}{(1-P^{-2})^2}
                                  =(1+P^{-1})^{-2}.
\]
The root scale contributes $A_{\mathbf j}^{-1/6}$,
and $(q/L)X^{1/6}=q^{5/6}L^{-2/3}$.
Finally,
\[
 \frac16\int_0^\infty\mathcal J_U(x)x^{-5/6}\,dx
     =K_w\II(U)^2,\qquad
 K_w=\frac16\left(\frac{2\pi}{\sqrt3}\right)^2
          \int_0^\infty\widehat w_{\rm rad}(t)t^{-5/6}\,dt.
\]
The last integral is positive. The Gaussian integral
representation of the Riesz kernel identifies it with
a positive constant times
$\int_0^\infty w_{\rm rad}(u)u^{-1/6}\,du>0$.

The unrestricted-row coefficient is therefore
\[
 C_{S,w}^{\,0}
  =\frac{\delta_0^2r_K K_w}{(\zeta_K^S(2))^2}
              \sum_{\epsilon,\mathbf r}
                         (\N e_S(\mathbf r))^{-1/6}>0,
\]
independent of $U,\sigma_{\rm ray},\mathfrak d,\mathbf e$.
Together with Lemma \ref{lem:nonexceptional},
\eqref{eq:exception-error}, the diagonal, and
\eqref{eq:mean-completion}, this proves the claimed
formula for $\mathcal H(q,L)$.

\subsection{Fixed row exclusions and short columns}

For $h$ supported on $S$, the coefficient
$\overline{\chi_s(h)}$ is constant of modulus one on
the column class. The identity
$G(s,-hm)=\overline{\chi_s(h)}G(s,-m)$ therefore shows
that its auxiliary Gram is exactly $\mathcal H(q,L)$.
Finite M\"obius inversion, and radiality of $w$, give
\[
 \|F_{\mathfrak d,\mathbf e;U}\|_q^2
       =\sum_{h\mid\prod_{p\in S}p}\mu(h)\mathcal H(q/\N h,L).
\]
The projected main coefficient is
\[
                  C_{S,w}=C_{S,w}^{\,0}
                                \prod_{p\in S}(1-P^{-5/6})>0.
\]
For $q/\N h\ge1$, apply the unrestricted formula.
For $q/\N h<1$ the nonzero row sum is empty; its proposed
main term is bounded and can be absorbed in the $D$ error.
The finite projector thus proves \eqref{eq:gram} for
all $q,L\ge1$ with $X\ge1$.

If $X\le1$, open each Gauss sum as a sum over reduced
fractions $x/s$ on the complex torus.
They have separation $\gg L^{-1}$ and total squared
coefficient mass $O_U(L^{-1})$.
Poisson summation and a packing bound give a Fourier
Gram row sum $O_w(q+L^2)$.
Schur's inequality therefore gives
$(q+L^2)O_U(L^{-1})\ll q/L$.
Moving twists have modulus at most one and the column
restrictions only remove coefficients.
The same upper bound holds for real signed profiles.

Every separation and counting error above uses finitely
many seminorms. The moment height order is fixed before
choosing external Fourier decay orders.
Hence the stated profile uniformity holds, completing
the proof of Proposition \ref{prop:gram}.

\section{Squarefree columns}
\label{sec:columns}

We next retain the squarefree column support using an exact
arithmetic decomposition common to both Gram estimates.
We first combine it with Proposition
\ref{prop:gram}; Appendix \ref{app:unbalanced} combines
it with the two-Poisson estimate instead.

\begin{lemma}[Prime-power Gauss factors]\label{lem:prime-power}
Let $p\notin S$, $P=\N p$, and $a\ge1$.
If $6\nmid a$, then $G(p^a,-m)$ vanishes unless
$v_p(m)=a-1$. For $m=p^{a-1}u$, $p\nmid u$, its
absolute value is $P^{(a-1)/2}$ and its dependence on $u$
is a character of modulus one.
If $6\mid a$, then
\[
 G(p^a,-m)=P^{a/2}1_{p^a\mid m}
                    -P^{a/2-1}1_{p^{a-1}\mid m}.
\]
\end{lemma}

\begin{proof}
Write a residue modulo $p^a$ as $x_0+py$.
Additive orthogonality forces $p^{a-1}\mid m$ and leaves
$P^{a/2-1}\sum_{x_0\bmod p}\chi_p(x_0)^a e(-ux_0/p)$.
For a nonprincipal local character this is a primitive
field Gauss sum, nonzero only for $p\nmid u$.
For a principal local character its values are $P-1$
and $-1$ according as $p\mid u$ or not.
\end{proof}

An ideal is powerful if every positive prime exponent
is at least two; the unit ideal is included.
Expanding $\mu^2(s)=\sum_{d^2\mid s}\mu(d)$ and assigning
to $t$ the entire part of $s$ supported on $d$ gives
an exact sum over powerful $t$, with coefficient $\mu(\rad t)$.
The remaining modulus $s_0=s/t$ is coprime to $\rad t$.
There is no added squarefree condition on $s_0$.
Gauss CRT gives
\[
 G(ts_0,-m)=\mathfrak r(t,s_0)\chi_{s_0}(t)^2G(t,-m)G(s_0,-m).
\]
The reciprocity factor is constant on the remaining
ray fiber. The normalization contributes $T^{-1}$,
where $T=\N t$.
An external $\mathfrak d$ remains coprime to $t$;
its exponents are unchanged, and its factor on $t$
is a scalar of modulus one.

At $p^a\parallel t$, Lemma \ref{lem:prime-power} gives
the following components. The squared scalar includes
the column normalization $P^{-a}$ before squaring:
\begin{equation}\label{eq:component-table}
 \begin{array}{c|c|c|c}
 \text{case}&\text{row divisor}&\text{squared scalar}&
                       \text{remaining twist exponent}\\ \hline
 6\nmid a&p^{a-1}&P^{-a-1}&a+1\pmod6\\
 6\mid a,\ \text{boundary}&p^{a-1}&P^{-a-2}&1\\
 6\mid a,\ \text{bulk}&p^a&P^{-a}(1-P^{-1})^2&0
 \end{array}
\end{equation}
The boundary and nonprincipal cases retain a unit
restriction on the quotient row until its squared
norm is bounded. In the bulk case that quotient is
unrestricted at $p$; no principal zero-extended row
character is inserted there.

Let $D_t=\N\rad t$, let $E$ be the product of norms
of the nonprincipal and boundary primes, and let
$E_{\rm bd}$ be the boundary product.
The row divisor has norm $T/E$.
After row division, the remaining parameters are
\begin{equation}\label{eq:component-scales}
 q_t=qE/T,\quad L_t=L/T,\quad X_t=X/(TE),\qquad
 \text{squared scalar}\le (TE E_{\rm bd})^{-1}.
\end{equation}
Radiality of $w$ justifies this scale change.
Extra quotient-row restrictions are removed only
after taking a nonnegative squared norm.
If $q_t<1$, the component is empty.
The support of the profile gives $T\le L$, so $L_t\ge1$.

\begin{lemma}[Uniform upper projection]\label{lem:projection-upper}
For nonnegative $U$ and $q,L\ge1$,
\begin{equation}\label{eq:sf-upper}
 \begin{split}
 \|F^{\rm sf}_{\mathfrak d,\mathbf e;U}(\,\cdot\,;L)\|_q
 \le{}&C_0\sqrt{C_{S,w}}\II(U)
       q^{5/12}L^{-1/3}\Omega(\mathfrak d,\mathbf e)^{1/2}\\
 &+O_\eps\left((2qLD)^\eps
       \{\sqrt{q/L}(1+X^{1/24})+\sqrt D\}\right).
 \end{split}
\end{equation}
Here
$C_0=\prod_{p\notin S}(1+f(\N p))$ with
$0\le f(P)<2P^{-3/2}$.
In particular $C_0-1<1/4$ can be arranged by a fixed
enlargement of $S$.
\end{lemma}

\begin{proof}
For a component with $X_t\ge1$, apply
Proposition \ref{prop:gram} with moving support $D D_t$.
The external factor $\Omega(\mathfrak d,\mathbf e)$
is unchanged.
Multiplying its main term by the exact scalar in
\eqref{eq:component-table} gives the following local
factors relative to the original main scale:
\begin{equation}\label{eq:local-main-factors}
 \begin{array}{c|c}
 a=6j+r,\ 1\le r\le5&(1+P^{-1})^{-2}P^{-7j-r-1}\\
 a=6j,\ \text{boundary}&(1+P^{-1})^{-2}P^{-7j-2}\\
 a=6j,\ \text{bulk}&(1+P^{-1})^{-2}(1-P^{-1})^2P^{-7j}.
 \end{array}
\end{equation}
For $a\ge2$ each is at most $P^{-a-1}$.
Their square roots sum to at most
\[
 f(P)\le\frac{P^{-3/2}}{1-P^{-1/2}}
              +\frac{P^{-7/2}}{1-P^{-3}}<2P^{-3/2}.
\]
The unit component contributes one.
The Euler product therefore bounds the entire main
contribution in the norm. Components with $X_t<1$
are bounded by the short-column estimate and may
also be included in this positive main upper sum.

The squared component error is at most
\begin{equation}\label{eq:component-error}
 (2qLD)^\eps E_{\rm bd}^{-1}
 \left[
 \frac qL\{T^{-1}+X^{1/12}T^{-13/12}E^{-1/12}\}
                         +\frac{D D_t}{TE}
 \right].
\end{equation}
Uniformity follows from
$q_tL_tD D_t=qLD\,D_tE/T^2\le qLD$.
The first term covers the short-column components.
After square roots the required sums are
\begin{equation}\label{eq:three-sums}
 \begin{split}
 \sum_{t,\mathrm{components}}T^{-1/2}E_{\rm bd}^{-1/2}
        &\ll_\eta L^\eta,\\
 \sum_{t,\mathrm{components}}
              T^{-13/24}E^{-1/24}E_{\rm bd}^{-1/2}
        &\ll1,\\
 \sum_{t,\mathrm{components}}
              (D_t/(TE E_{\rm bd}))^{1/2}
        &\ll_\eta L^\eta .
 \end{split}
\end{equation}
For the first sum, a powerful ideal is $u^2v^3$ with
$v$ squarefree, and divisor multiplicities cost a
small power. The second has leading prime factor
$P^{-9/8}$ and converges.
For the third the local factors are
$P^{-a/2}$, $P^{-(a+1)/2}$, and $P^{-(a-1)/2}$
in the three cases of \eqref{eq:component-table}.
Multiplication by $P^{-a\eta}$ makes the leading
$a=2$ term $P^{-1-2\eta}$ and gives a convergent
Euler product. This proves its truncated bound.
Minkowski and $\sqrt{x+y}\le\sqrt x+\sqrt y$ now
give \eqref{eq:sf-upper}.
\end{proof}

\begin{corollary}[A squarefree-column starting vector]
\label{cor:column-anchor}
Fix $0<\delta<1$. After a fixed enlargement of $S$,
uniformly for $B\le A\le B^{2-\delta}$ and sufficiently
large $B$,
\[
 \|F^{\rm sf}_{1,0;W}(\,\cdot\,;B)\|_A^2
                                 \gg A^{5/6}B^{-2/3}.
\]
The constant is positive for every fixed nonnegative,
nonzero $W\in C_c^\infty((1/2,1))$.
\end{corollary}

\begin{proof}
Put $Y=B^2/A$ and $\mathcal M=A^{5/6}B^{-2/3}$.
Proposition \ref{prop:gram} gives the all-modulus norm
\[
             (\sqrt{C_{S,w}}\II(W)+o(1))\sqrt{\mathcal M}.
\]
The nonunit powerful components change this norm by
at most
\[
 (C_0-1)\sqrt{C_{S,w}}\II(W)\sqrt{\mathcal M}
 +O_\eps\left((AB)^\eps
                     \{\sqrt{A/B}(1+Y^{1/24})+1\}\right).
\]
The error divided by $\sqrt{\mathcal M}$ is
\[
 O_\eps\left((AB)^\eps
             \{Y^{-1/12}+Y^{-1/24}+\mathcal M^{-1/2}\}\right)=o(1),
\]
since $Y\ge B^\delta$, $\mathcal M\ge B^{1/6}$,
and $AB\le B^3$.
Choose the source power losses below these margins,
and choose the fixed prime cutoff with $C_0-1<1/4$.
The triangle inequality gives the asserted positive lower bound.
\end{proof}

\section{The optimal full-cutoff lower bound}
\label{sec:optimal-proof}

The previous section provides a squarefree-column
Gauss vector with large energy on unrestricted rows.
The remaining issue is to retain squarefree rows
without losing the coefficient structure needed for
a change of column scale. We first prove the centered
comparison, then isolate a small-twist starting vector,
and finally transport its energy.

\subsection{Canceling the profile main term}

The Gram main term has squared size proportional to
$L^{-2/3}$. Multiplying the test sum by $L^{1/3}$
removes this column-scale dependence. We compare these
normalized sums by choosing profiles with the same
value of $\II$, so that their common main term cancels.

Choose a fixed nonnegative, nonzero
$W\in C_c^\infty((3/4,7/8))$.
For $1\le\rho\le4/3$ put
\[
               W_\rho(t)=\rho^{1/6}W(\rho t),\qquad
               V_\rho(t)=W(t)-W_\rho(t).
\]
All three nonnegative profiles $W,W_\rho,W+W_\rho$
are supported in $(1/2,1)$ with uniform seminorms.
Also $\II(W_\rho)=\II(W)$.
In the next identity, suppress the fixed arithmetic
subscripts and the common column scale in $F$.
Linearity and the parallelogram identity give
\[
 \|F_{V_\rho}\|_q^2
  =2\|F_W\|_q^2+2\|F_{W_\rho}\|_q^2
                                -\|F_{W+W_\rho}\|_q^2.
\]
The main terms in \eqref{eq:gram} cancel exactly.
This derives the signed-profile bound from nonnegative
profiles.

\begin{lemma}[Centered scale comparison]\label{lem:scale}
Fix the moving data $\mathfrak d,\mathbf e$ and row scale $q$.
If $1\le L_-\le L_+$ and $q,L_+,D$ are in fixed-power
ranges of $A\ge2$, then for every $\zeta>0$,
\begin{equation}\label{eq:scale-comparison}
 \begin{split}
 &\big\|L_+^{1/3}F^{\rm sf}_{\mathfrak d,\mathbf e;W}
                              (\,\cdot\,;L_+)
      -L_-^{1/3}F^{\rm sf}_{\mathfrak d,\mathbf e;W}
                              (\,\cdot\,;L_-)\big\|_q^2\\
 &\hspace{15mm}\ll_\zeta A^\zeta
       \{qL_-^{-1/3}+q^{11/12}L_-^{-1/6}+D L_+^{2/3}\}.
 \end{split}
\end{equation}
The same bound holds on any fixed subset of the row set.
\end{lemma}

\begin{proof}
First impose squarefree columns on the centered profile.
Every powerful component has the same profile $V_\rho$,
so its full Gram main term is zero by the preceding
parallelogram identity.
Use \eqref{eq:component-error} and all three sums in
\eqref{eq:three-sums}. Empty row components are omitted;
the short-column bound covers $X_t<1$.
This gives
\[
 \|F^{\rm sf}_{\mathfrak d,\mathbf e;V_\rho}
                                  (\,\cdot\,;L)\|_q^2
 \ll_\eps(2qLD)^\eps\{(q/L)(1+X^{1/12})+D\}.
\]
Uniformity in $\rho$ follows from the common seminorm bounds.

The exact normalization in \eqref{eq:gauss-polynomial} gives
\[
 L^{1/3}F^{\rm sf}_W(m;L)
 -(L/\rho)^{1/3}F^{\rm sf}_W(m;L/\rho)
                         =L^{1/3}F^{\rm sf}_{V_\rho}(m;L).
\]
After multiplication by $L^{2/3}$, the three error terms
are $qL^{-1/3}$, $q^{11/12}L^{-1/6}$, and $D L^{2/3}$.
Join the endpoints by $O(\log A)$ adjacent steps of
ratio at most $4/3$. The first two terms decrease with
$L$ and the third increases.
Minkowski, followed by a rechoice of the power loss,
proves \eqref{eq:scale-comparison}.
Restricting the outer squared norm of the difference
preserves this upper bound.
\end{proof}

\subsection{Two tail estimates using only known upper bounds}

The next lemma uses \eqref{eq:theta6} and \eqref{eq:xi3}.
It is shared by the present structured-vector argument
and the original operator-norm transfer in
Appendix \ref{app:unbalanced}.

For squarefree-column coefficients $\lambda$, write
$P_\lambda(u)=\sum_b\lambda_bJ_6(u,b)$.
Fix $0<\delta<1/3$ and
$B^{4/3+\delta}\le A\le B^{2-\delta}$.
Put $Z=AB$, $M=A^{5/6}B^{1/3}$, and $g=\delta/18$.

\begin{lemma}\label{lem:tails}
For every $\rho>0$, the sixth-power decomposition
$m=t^6u$ has tail
\begin{equation}\label{eq:t6-tail}
 \sum_{\substack{\N t>T_6\\\N(t^6u)\le A\\u\ {\rm sixth\!-\!power\!-\!free}}}
       |P_{\lambda^{(t)}}(u)|^2
 \ll Z^\rho M(T_6^{-4}+Z^{-g})\|\lambda\|_2^2,
\end{equation}
where $\lambda_b^{(t)}=\lambda_b1_{(b,t)=1}$.
For
$u=a_1a_2^2a_3^3a_4^4a_5^5$ with pairwise coprime
squarefree factors, the higher-factor tail satisfies
\begin{equation}\label{eq:higher-tail}
 \sum_{\substack{\N u\le A\\\prod_{j=2}^5\N a_j>T}}
                    |P_\lambda(u)|^2
       \ll Z^\rho M(T^{-2/3}+Z^{-g})\|\lambda\|_2^2.
\end{equation}
\end{lemma}

\begin{proof}
The identity $J_6(t^6u,b)=J_6(u,b)1_{(b,t)=1}$
does not require $(t,u)=1$.
Apply \eqref{eq:theta6} at row length $A/(\N t)^6$
and sum $\N t>T_6$. The four terms give
\[
 Z^\rho\{AT_6^{-5}+BA^{1/6}+MT_6^{-4}
                    +A^{1/3}B^{5/6}T_6^{-1}\}\|\lambda\|_2^2.
\]
In the indicated range this is bounded as in
\eqref{eq:t6-tail}.

For the second assertion normalize $\|\lambda\|_2=1$.
Fix dyads $\N a_j\asymp F_j$, $j\ge2$, and set
$F=\prod_{j=2}^5F_j^j$, $S_1=\prod_{j=2}^5F_j$.
Fixing the higher factors and averaging $a_1$ gives
$U_1\ll S_1\Xi_6(A/F,B)$.
Fixing $a_1,a_3,a_4,a_5$ and averaging $a_2$ gives
$U_2\ll (A/F)(S_1/F_2)\Xi_3(2F_2,B)$.
The fixed factors only twist or remove coefficients,
without increasing their norm.
Remaining row coprimality restrictions may be dropped
in these nonnegative upper bounds.

Using $F\ge S_1^2$, \eqref{eq:theta6} gives
\[
 U_1\ll Z^\rho\{A/S_1+BS_1+MS_1^{-2/3}
                         +A^{1/3}B^{5/6}S_1^{1/3}\}.
\]
For $T<S_1\le A^{1/3}$ this is at most
$Z^\rho\{A/T+MT^{-2/3}+A^{1/3}B+A^{4/9}B^{5/6}\}$.
For $S_1>A^{1/3}$, the inequalities
\[
 S_1/F\le S_1^{-1},\quad
 S_1/(FF_2)\le S_1^{-2},\quad
 S_1/(FF_2^{1/3})\le S_1^{-4/3}
\]
and \eqref{eq:xi3} give
$U_2\ll Z^\rho\{A+A^{1/3}B+A^{5/9}B^{2/3}\}$.
The terms independent of $T$, divided by $M$, are bounded by
\[
 B^{-\delta/6},\quad B^{-\delta/2},\quad
 B^{-1/54-7\delta/18},\quad B^{-1/27-5\delta/18},
\]
respectively. Since $Z\le B^3$, each is $O(Z^{-g})$.
Also $A/T\le MT^{-2/3}$.
There are only $O(\log^4(2A))$ dyads.
Actual and dyadic product cutoffs differ by a fixed
factor; use a smaller preliminary power loss to absorb
these factors and the logarithms.
This proves \eqref{eq:higher-tail}, following the
factor-selection method of \cite[Lemma 3.4]{BGL}.
\end{proof}

\subsection{A structured squarefree-row starting point}

Set $L_0=A^{2/3}$, $Z=AL_0$, and fix
$\delta_*=1/44$, $g=\delta_*/18$.
The pair $(A,L_0)$ is in the range of Lemma \ref{lem:tails}.
By Corollary \ref{cor:column-anchor}, the squarefree-column
Gauss vector has weighted all-row energy
$\gg\mathcal M=A^{5/6}L_0^{-2/3}$.

\begin{lemma}\label{lem:structured-anchor}
For $0<\eta\le g/100$ and sufficiently large $A$,
there are a unit $\epsilon$ and a primary $h$ prime to $S$,
with $\N h\ll Z^{21\eta}$, such that, for $q=A/\N h$,
\begin{equation}\label{eq:structured-anchor}
 \sum_{\substack{a\ {\rm primary,\ squarefree}\\(a,S)=1}}
 w(a/\sqrt q)
       |L_0^{1/3}F^{\rm sf}_{1,0;W}(\epsilon h a;L_0)|^2
                           \gg_\eta A^{5/6}Z^{-5\eta}.
\end{equation}
\end{lemma}

\begin{proof}
Split literal rows into six unit classes and choose one
carrying at least a sixth of the energy.
Primitive Gauss evaluation identifies its conjugate
polynomial with $P_\lambda(m)$, where
\[
 \lambda_b=L_0^{-1}W(q_b/L_0)(q_b/L_0)^{-1/2}
       \overline{G(b,1)}\chi_b(-\epsilon)
                   1_{b\ {\rm sf},\ (b,S)=1,\ [b]=\sigma_{\rm ray}}.
\]
It has squared coefficient norm $O(L_0^{-1})$ and
weighted row energy $\gg\mathcal M$.

Take $\rho=\eta/4$, $T_6=Z^\eta$ in
\eqref{eq:t6-tail}. The tail is $o(\mathcal M)$.
One of the $O(Z^\eta)$ remaining ideals $t$ thus has
weighted energy $\gg\mathcal M Z^{-\eta}$.
Keep the original weight in this head.
For its fixed vector $\lambda^{(t)}$, apply
\eqref{eq:higher-tail} at the larger cutoff $A$ and
$T=Z^{3\eta}$. The weighted tail is at most
\[
 O\bigl(\mathcal M
       \{Z^{-7\eta/4}+Z^{\eta/4-g}\}\bigr),
\]
which is smaller than the retained energy.
There are $O(T\log^3(2T))=O_\eta(Z^{4\eta})$
remaining tuples $(a_2,a_3,a_4,a_5)$.
One carries weighted energy $\gg_\eta\mathcal M Z^{-5\eta}$.

Set $h=t^6a_2^2a_3^3a_4^4a_5^5$.
Then $\N h\ll Z^{6\eta}T^5=Z^{21\eta}$.
Its coefficient vector is exactly $\lambda_b\chi_b(h)$,
including every zero at primes of $t$.
Dropping the remaining coprimality conditions on the
squarefree $a_1$ enlarges its positive row sum.
Radiality leaves precisely $w(a_1/\sqrt{A/\N h})$.
Multiplication by $L_0^{2/3}$ proves
\eqref{eq:structured-anchor}.
\end{proof}

\subsection{Transport and the full parameter range}

\begin{proof}[Proof of Theorem \ref{thm:optimal}]
First assume $A\ge B$ and
$A^{1/2+\vartheta}\le B\le A$ for a fixed
$0<\vartheta<1/6$.
Choose
\[
 0<\eta\le\min\{\eps/100,\vartheta/100,g/100,1/1000\},
\]
and select $\epsilon,h,q$ by Lemma \ref{lem:structured-anchor}.
In particular $q\le A$ and
$D=\N\rad h\le\N h\ll A^{35\eta}$.
For squarefree columns the exact numerator identity is
\[
 F^{\rm sf}_{1,0;W}(\epsilon h a;L)
       =F^{\rm sf}_{\rad h,\mathbf e;W}(\epsilon a;L),
                     \qquad e_p=(-v_p(h)\bmod6).
\]
Primes dividing $h$ are excluded from the columns;
even a six-divisible exponent retains this mask.

On the same squarefree primary rows with weight
$w(a/\sqrt q)$, put
$H_L(a)=L^{1/3}F^{\rm sf}_{1,0;W}(\epsilon h a;L)$.
The unit change of literal row preserves the full
radial norm. Lemma \ref{lem:scale}, followed by
restriction to this row subset, yields
\[
 \|H_B-H_{L_0}\|_{\rm sf,q}^2
 \ll_\zeta A^{5/6-\vartheta/6+\zeta}
                       +A^{2/3+35\eta+\zeta}.
\]
Here both endpoints lie between $A^{1/2+\vartheta}$ and $A$.
Take the source loss small enough that $\zeta=\eta$,
including the telescoping logarithms.
The anchor energy is
$\|H_{L_0}\|_{\rm sf,q}^2\gg A^{5/6-25\eta/3}$.
The differences between this exponent and the two
error exponents are
\[
             \vartheta/6-28\eta/3>0,\qquad
                         1/6-133\eta/3>0 .
\]
They are at least $11\vartheta/150$ and $367/3000$,
respectively.
Thus the triangle inequality in this same row space
gives $\|H_B\|_{\rm sf,q}^2\gg A^{5/6-25\eta/3}$.

At the new column length the coefficient vector is
\[
 \lambda_b'=B^{-1}W(q_b/B)(q_b/B)^{-1/2}
       \overline{G(b,1)}\chi_b(-\epsilon h)
                    1_{b\ {\rm sf},\ (b,S)=1,\ [b]=\sigma_{\rm ray}}.
\]
It has squared norm $O(B^{-1})$.
It is nonzero by the retained energy.
Since $q=A/\N h\le A$ and $w\le1$, its Rayleigh quotient gives
\[
                  \Xi_{6,S}(A,B)
                       \gg A^{5/6-25\eta/3}B^{1/3}.
\]
The loss is at most $\eps/12$.

For all $A\ge B$, take
$\vartheta=\min\{\eps,1/12\}$.
If $B<A^{1/2+\vartheta}$, then
$A^{5/6}B^{1/3}\le A^{1+\vartheta/3}$.
The unit column and the positive-density count of
squarefree ideals give $\Xi_{6,S}(A,B)\gg_S A$,
which suffices with $(AB)^{-\eps}$.
Otherwise use the preceding argument.
The other mixed term is no larger for $A\ge B$,
and $B\le A$, so this proves the whole sum.

For $A<B$, split the matrix into its finitely many
row and column ray blocks.
Reciprocity \eqref{eq:reciprocity} identifies each
transposed block up to a scalar of modulus one.
Transposition preserves the spectral norm, and the
full norm is bounded by the sum of the block norms.
Thus $\Xi_{6,S}(A,B)$ and $\Xi_{6,S}(B,A)$ are
comparable up to a fixed factor.
Finally use \eqref{eq:submatrix}.
Bounded parameter ranges are covered by the unit
row and column after decreasing the constant.
At balance, rename the loss to obtain the stated
$Q^{7/6-\eps}$ estimate.
\end{proof}

\begin{remark}
The proof compares the column scales of one small-twist
Gauss coefficient formula, bounding the centered difference
on all rows before restricting to squarefree rows.
The selected row annulus is $A/\N h$, which may differ
from $A$ by a small power.
\end{remark}

\section{Zero density from unmarked moments}
\label{sec:zero-density}

We now begin the structured-coefficient branch.
In this and the next section, $S$ may be any fixed
finite set containing the primes above $6$.
We keep the original presentations
$\psi_u(n)=(u/n)_6$ and their collision zeros.
Starting from the plain moment in Section \ref{sec:setup},
we add the inverse moment and buffered detector, count
rows in zero-location bins, and prove the density bound.
Section \ref{sec:prime-moments} then converts those bin
counts into estimates for prime coefficients.

\subsection{The inverse moment and buffered detector}

The detector converts a far-right zero into simultaneous
large values of a M\"obius-weighted polynomial and a plain
character polynomial. Their moment bounds then limit
the number of rows on which this can occur.

Fix a finite group $\Theta$ of ray characters supported
on $S$, containing the required fixed twists.
For a physical row $u$, let
\[
 \mathcal X_u=
 \{\nu(n)\chi_n(u)^\epsilon:
             \nu\in\Theta,\ \epsilon\in\{-1,1\}\}.
\]
This finite family collects the fixed twists and the
two orientations needed by the detector on the same row.
The inverse orientation is still zero on nonunits.
For $\psi\in\mathcal X_u$, let $Q_\psi$ be the norm of
the conductor of $\psi^*$, and let $E_u$ be the set of
Euler primes removed from its primitive $L$-function
by the original zero extension. Then
\begin{equation}\label{eq:original-presentation}
 Q_\psi\ll_S Q,\qquad
 L_{\rm orig}(s,\psi)=L(s,\psi^*)
          \prod_{p\in E_u}(1-\psi^*(p)(\N p)^{-s}),
       \qquad \N\!\prod_{p\in E_u}p\ll_S Q .
\end{equation}
This follows from reciprocity and the fixed $S$-data.
Every deleted factor is nonzero for $\Rea s>0$.
On a nonunit sixth-power-free row outside $S$, a
good-prime valuation $j\in\{1,\ldots,5\}$ has local
order $6/\gcd(6,j)>1$. A fixed $S$-twist cannot
cancel it. Thus every presentation in this family
induces a nonprincipal character.

\begin{lemma}[Unmarked inverse estimate, \cite{OAI}]
\label{lem:inverse-input}
For $r$ in a prescribed bounded nonnegative range, put
\[
 M_\psi(r;V)=Q^{-r/2}
             \sum_n\mu(n)\psi(n)V(\N n/Q^r).
\]
For each fixed finite presentation label,
\begin{equation}\label{eq:inverse-moment}
 \sum_{u\in\mathscr U_S(Q)}|M_{\psi_u}(r;V)|^2
 \ll_\eta Q^{E(r)+\eta}(1+H)^C,\qquad
 E(r)=\max\{1,(1+5r)/6\}.
\end{equation}
Here $H$ bounds the separated norm heights.
The constant uses finitely many smooth seminorms.
The permitted rowwise real powers, heights, and smooth
scale parameters have a fixed polynomial height cost.
\end{lemma}

The polynomial $M_\psi$ is a smooth, normalized segment
of the M\"obius-weighted Dirichlet series for the reciprocal
of $L_{\rm orig}(s,\psi)$.
The inverse estimate is \cite[Lemma 17.6]{OAI}, including its last
uniformity assertion and the parameter Sobolev argument
of its Lemma 4.5.
For the plain polynomial, retain the length convention
of Theorem \ref{thm:plain-input} and take $N=Q^m$:
\[
             S_\psi(Q^m;V)=Q^{-m/2}\sum_n\psi(n)V(\N n/Q^m).
\]
Theorem \ref{thm:plain-input}, with no moving extra twist,
gives
\begin{equation}\label{eq:plain-fourth}
             \sum_{u\in\mathscr U_S(Q)}
                    |S_{\psi_u}(Q^m;V)|^4
                    \ll_\eta Q^{1+\eta}(1+H)^C .
\end{equation}
Both estimates retain the same allowed presentations and
zero extensions.

We group rows by the rightmost zero visible in a sequence
of height windows. The label $a$ records its real-part
bin, and $i$ selects a window whose enlargement has
controlled rightmost zero. This buffer will provide
room for the finite-height contour shift in the
prime-sum estimate.

\begin{lemma}[Buffered bins and simultaneous witnesses]
\label{lem:detector-input}
Let $a_0=51/100$, choose a sufficiently small fixed
$0<e<10^{-3}$, and put $H_1=Q^\tau$,
$0<\tau\le1/100$. Each row has one of $O_e(1)$ bins
$(i,a)$ with
$a\in(a_0+e\mathbb Z_{\ge0})\cap[a_0,1]$.
For $I=\lceil1/e\rceil+2$, define
\[
 M_j(u)=\max\bigl(\{a_0\}\cup
 \{\Rea\rho:\psi\in\mathcal X_u,\ L(\rho,\psi^*)=0,\
       \Rea\rho\ge a_0,\ |\Ima\rho|\le3jH_1\}\bigr).
\]
The assigned $1\le i<I$ satisfies
\begin{equation}\label{eq:bin}
                   a\le M_i(u)<a+e,\qquad M_{i+1}(u)<a+2e.
\end{equation}
Let $\mathcal B_{i,a}$ be the set of rows
$u\in\mathscr U_S(Q)$ assigned to $(i,a)$.
Bins with $a=a_0$ are called floor bins; those with
$a>a_0$ are above-floor bins.
For an above-floor bin put $\delta=2a-1$.
For every $t\in[1,3/2]$ and arbitrarily small prescribed
loss $\eta>0$, the source detector supplies a common
presentation $\psi\in\mathcal X_u$ and dyadic witnesses
$M_\psi(r;V_M)$, $S_\psi(Q^m;V_S)$ with
\begin{equation}\label{eq:witnesses}
 \begin{gathered}
 r\le t+O(1/\log Q),\qquad r+m\ge t-O(1/\log Q),\\
 0\le m\le\tfrac12+O(\eta),\qquad
 r\ge t-\tfrac12-O(\eta),\\
 |M_\psi(r;V_M)|^2\gg Q^{\delta r-\eta},\qquad
 |S_\psi(Q^m;V_S)|^2\gg Q^{\delta m-\eta}.
 \end{gathered}
\end{equation}
Their profiles have the finite-dimensional smooth dependence
specified in the source and heights $O_e(H_1)$.
There are $O((\log Q)^2)$ dyadic pairs and finitely many
presentation labels.
\end{lemma}

The assertions are \cite[Lemmas 8.1--8.2 and Proposition 8.3]{OAI}.
The bin construction takes successive maxima in increasing
height windows.
The maxima are nondecreasing and bounded by one; two
successive values differ by at most $e$, giving
\eqref{eq:bin} after rounding.
If all zeros lie below the floor, all rows are floor rows
and no above-floor witness is needed.

\paragraph{Order of choices.}
First fix the requested final power loss, preliminary
moment and detector losses, and a sufficiently small $e$.
Fix the finite moment, derivative, and parameter-Sobolev
orders. Then choose $\tau>0$ so small that all fixed
powers of $1+O_e(H_1)$ fit into the reserved power loss
and the detector's height hypotheses hold.
Only after $\tau$ is fixed are external Mellin-tail
orders chosen. Increasing those orders differentiates
external profiles and does not increase the already
fixed internal height order.

\subsection{Counting rows}

The two witnesses give competing upper bounds for the
number of rows in a bin. Optimizing these bounds over
their lengths produces the decreasing count exponent
\[
 \RR(a)=
 \begin{cases}
 \displaystyle\frac{31-32a}{21-12a},&a_0<a\le11/12,\\[4pt]
 2(1-a),&11/12\le a\le1.
 \end{cases}
\]
Its first branch agrees with $R(\sigma)$ in
Theorem \ref{thm:prime-density} when $a=\sigma$.
The extension to larger $a$ is needed when summing
the contributions of all bins to the prime moment.

\begin{proposition}\label{prop:bad-rows}
Every above-floor bin satisfies
$\#\mathcal B_{i,a}\ll_\eps Q^{\RR(a)+\eps}$.
Floor bins have $O(Q)$ rows.
\end{proposition}

\begin{proof}
Fix $t$ and partition by the finite presentation label
and dyadic witness lengths. Chebyshev, applied separately
to \eqref{eq:inverse-moment} and \eqref{eq:plain-fourth},
gives the count exponent
\[
                       \min\{E(r)-\delta r,\ 1-2\delta m\},
\]
up to arbitrarily small losses and a fixed height power.
For $0<\delta\le5/6$ and $r<1$, use $m\ge t-r-o(1)$.
The largest possible value of
$\min\{1-\delta r,1-2\delta(t-r)\}$ is
$1-2\delta t/3$, at $r=2t/3$.
For $r\ge1$, the inverse estimate is at most
$1/6+(5/6-\delta)t$.
Choose $t_*=5/(5-2\delta)\in[1,3/2]$.
The two branches agree at
\[
                \frac{15-16\delta}{15-6\delta}
                              =\frac{31-32a}{21-12a}.
\]
For $5/6\le\delta\le1$, choose $t=3/2$.
The witness has $r\ge1-O(\eta)$, so the inverse
estimate is at most $1-\delta+O(\eta)$.
All endpoint errors are absorbed by first decreasing
the preliminary losses. The indicated functions of
the lengths are uniformly Lipschitz.
Finally absorb the dyadic and height costs in the
prescribed order of choices.
\end{proof}

The first branch is exactly $R_\kappa(\delta,0)$
in the detected-row optimization of \cite{Sang}.
The proof above uses its unmarked balance.
The second branch disposes of high bins.

\subsection{Zero density, multiplicity, and cumulative cutoffs}

\begin{proof}[Proof of the density assertions in Theorem \ref{thm:prime-density}]
Choose $e<(\sigma-a_0)/2$.
For large $Q$, $H_1=Q^\tau$ exceeds $(\log Q)^C$.
If $L(s,\psi_u^*)$ has a zero counted in
\eqref{eq:density-def}, then $M_i(u)\ge\sigma$,
because that presentation is included in $\mathcal X_u$.
Thus its bin satisfies $a>\sigma-e>a_0$.
Proposition \ref{prop:bad-rows} and monotonicity of
$\RR$ bound the number of such rows by
$Q^{R(\sigma)+O(e)+\eta}$, after reserving the height loss.

The standard zero count for a primitive finite-order
Hecke function over this fixed field is
\[
 \#\{\rho:0\le\Rea\rho\le1,\ |\Ima\rho|\le T\}
       \ll(T+1)\log(2Q_\psi(T+3)^2);
\]
see \cite{IK}. It counts multiplicities.
By \eqref{eq:original-presentation}, this is
$O_C((\log Q)^{C+1})$ in the present range.
Multiplication gives the asserted density bound.
Deleted Euler factors have no zeros in this half-plane.

For the cumulative version put $R=R(\sigma)>0$
and fix $\theta=R/2$.
On dyads $U\ge Q^\theta$,
$T\le\theta^{-C}(\log U)^C$; the same proof admits
this fixed multiplicative constant.
Sum the resulting bounds $U^{R+\eta}$.
For $1<\N u<Q^\theta$, the ideal count and the
individual zero bound give
$O(Q^\theta(\log Q)^{C+1})$.
Both contributions are within the claimed bound.
\end{proof}

\section{Structured prime moments}
\label{sec:prime-moments}

The zero-location bins carry two pieces of information:
there are few rows in a high bin, and the buffered
rectangle bounds the prime sum on each remaining row.
Balancing these two estimates gives the exponent
$\kappa_6$. Throughout, the prime weights are fixed
smooth profiles.

\subsection{A prime sum inside one bin}

\begin{lemma}\label{lem:prime-bin}
For every row in a bin $(i,a)$, including a floor bin,
and every fixed $V\in C_c^\infty((1,2))$,
\[
 \left|\sum_{\pi\ {\rm primary\ prime},\ (\pi,S)=1}
                  V(\N\pi/Q)\psi_u(\pi)\right|
                            \ll_{V,e,\eta}Q^{a+8e+\eta}.
\]
\end{lemma}

\begin{proof}
Put $\alpha=a+8e$ and $H=H_1/4$.
The buffered disks in \eqref{eq:bin}, the logarithmic
control of \cite[Lemma 4.9]{OAI}, and the deleted-factor
bound in its Lemma 4.10 give
\begin{equation}\label{eq:logder}
 \left|\frac{L_{\rm orig}'}{L_{\rm orig}}(s,\psi_u)\right|
       \ll_e\log Q
 \quad(\alpha\le\Rea s\le2,\ |\Ima s|\le H).
\end{equation}
The primitive conductor and the deletion radical both
have norm $O_S(Q)$.

Set $V_Q(x)=V(x)/(\log Q+\log x)$.
Each fixed logarithmic derivative is $O_V(1/\log Q)$.
On $\Rea s=2$, Mellin inversion gives
\[
 \sum_n\Lambda_K(n)\psi_u(n)V_Q(\N n/Q)
  =\frac1{2\pi i}\int_{(2)}
       -\frac{L_{\rm orig}'}{L_{\rm orig}}(s,\psi_u)
                          Q^s\widehat V_Q(s)\,ds.
\]
The prime terms are exactly the desired sum.
The terms $n=\pi^k$, $k\ge2$, have coefficient
$V(\N\pi^k/Q)/k$ and contribute
$O_V(Q^{1/2}\log Q)$ by absolute counting.

Move only the segment of height at most $H$ to
$\Rea s=\alpha$. The row is nonprincipal, and no
zero or deleted-factor zero occurs in this rectangle.
By \eqref{eq:logder}, the shifted segment is
$O_{V,e}(Q^\alpha)$.
For each fixed $J$, the horizontal joins are
$O(Q^2H^{-J})$ and the two unshifted absolute-line
tails are $O(Q^2H^{1-J})$.
Choose $J>1+12/\tau$ after $\tau$, making these
errors $O(Q^{-10})$.
The external Mellin profile is differentiated
to this order. Since $a\ge51/100>1/2$, the prime-power
error is absorbed as well.
\end{proof}

\subsection{The second moment and its optimal bin}

\begin{proof}[Proof of the prime-moment assertion in
Theorem \ref{thm:prime-density}]
First sum over all sixth-power-free rows in $\mathscr U_S(Q)$.
On an above-floor bin, Proposition \ref{prop:bad-rows}
and Lemma \ref{lem:prime-bin} give exponent
$\RR(a)+2a+16e+O(\eta)$.
For $a\le11/12$ put
$F(a)=(31-32a)/(21-12a)+2a$.
Then
\[
 F'(a)=2-\frac{300}{(21-12a)^2},\qquad
 F''(a)=-\frac{7200}{(21-12a)^3}<0.
\]
The maximum occurs at
\[
 a_*=\frac{21-5\sqrt6}{12},\qquad
 F(a_*)=\frac{37-10\sqrt6}{6}=\kappa_6 .
\]
For $a\ge11/12$, $\RR(a)+2a=2$.
The floor bins use the trivial row count and give
$1+2a_0=2.02$ before the small losses.
Both are below $\kappa_6$.
Choose all losses in the prescribed order and sum
the finite set of bins.
Finally restrict the nonnegative outer squared sum
to squarefree rows.
\end{proof}

\subsection{Bounded rough prime convolutions}

\begin{proposition}\label{prop:prime-convolution}
Fix $c>0$ and $k_0\ge1$.
For $1\le k\le k_0$ and $\prod_{j=1}^k F_j\asymp Y$, let
\[
 \beta_0(b)=
 \sum_{b=\pi_1\cdots\pi_k}
                   \prod_{j=1}^k w_j(\N\pi_j/F_j),
                 \qquad \beta(b)=\mu^2(b)\beta_0(b).
\]
The primes are primary and outside $S$, the smooth
annular weights are independently specified, and every
supported prime has norm at least $Y^c$.
Assume every fixed derivative is bounded by a fixed
power of $\log Y$, uniformly over the prescribed family.
Then
\begin{equation}\label{eq:prime-convolution-moment}
 \sum_{\substack{u\ {\rm sixth\!-\!power\!-\!free}\\
                         \N u\asymp Y,\ (u,S)=1}}
      \left|\sum_b\beta(b)(u/b)_6\right|^2
       \ll_\eps Y^{\kappa_6+\eps}+Y^{13/6-c+\eps}.
\end{equation}
Constants may depend on the fixed data, $c,k_0$, and
the prescribed support and derivative bounds.
\end{proposition}

\begin{proof}
Without the squarefree restriction the polynomial is
exactly the product of the $k$ prime sums.
On a buffered bin, the proof of Lemma \ref{lem:prime-bin}
at each length $F_j$ gives $F_j^{a+O(e)}Y^\eta$.
Nonempty factors have $Y^c\ll F_j\ll Y$.
Their number is bounded, their external seminorms
cost only logarithms, and all Mellin tails are handled
in the same finite height rectangle.
The product is therefore bounded by $Y^{a+O(e)+O(\eta)}$.
The same bin count and optimization give a squared
row norm $O(Y^{\kappa_6+\eps})$.

Let $\Delta=(1-\mu^2)\beta_0$.
Its coefficients are divisor-bounded up to logarithms,
and every nonzero coefficient contains a repeated
prime of norm at least $Y^c$.
The support has size
\[
               O\left(\sum_{\N\pi\ge Y^c}
                                  Y/(\N\pi)^2\right)=O(Y^{1-c}).
\]
For the sixth-power-free columns in this discrepancy,
\eqref{eq:theta6} gives squared row norm
$O(Y^{13/6-c+\eps})$.

Columns that are not sixth-power-free require a separate
reduction. Write $b=b_0t^6$ uniquely, with $b_0$
sixth-power-free; no condition $(b_0,t)=1$ is imposed.
The exact identity is
$(u/b)_6=(u/b_0)_6\,1_{(u,t)=1}$.
For fixed $t$, apply \eqref{eq:theta6} after dropping
only this outer row restriction.
The row norm, before squaring, is at most
\[
 Y^{7/12+\eta}(Y/(\N t)^6)^{1/2}
                         =Y^{13/12+\eta}(\N t)^{-3}.
\]
If $t\ne1$ contributes, a supported prime divides it,
so $\N t\ge Y^c$.
Minkowski and ideal counting bound the squared norm
of their sum by
\[
 \left(Y^{13/12+\eta}
                   \sum_{\N t\ge Y^c}(\N t)^{-3}\right)^2
                                \ll Y^{13/6-4c+\eps}.
\]
Nonempty unit-scale ranges and fixed annular support
constants change only the implied constants.
Combining these estimates proves
\eqref{eq:prime-convolution-moment}.
\end{proof}

This coefficient class is the one used in
\cite[Section 4]{Patterson}.
In particular, the estimate saves any fixed power
smaller than $\min\{13/6-\kappa_6,c\}$ over the general
sextic comparison. For coprime squarefree $r,s$ of norm
$\asymp Y^{1/3}$, the map $(r,s)\mapsto rs^2$ is
injective with sixth-power-free image of norm $\asymp Y$,
so the proposition also applies to that mixed-order
row configuration with these same structured coefficients.

\appendix
\section{The original unbalanced argument without the plain moment}
\label{app:unbalanced}

This appendix retains the first version's result and its
distinct analytic route. It uses the finite Gauss and
reciprocity identities, the elementary complete-mean
calculation, and the known upper bounds in Section
\ref{subsec:known-sieves}.

\begin{theorem}[The first-version bounds]
\label{thm:unbalanced}
Fix a finite set $S_0$ containing the primes above $6$.
For $0<\delta<3/11$, $\eps>0$, and $A,B\ge2$ with
$B^{16/11+\delta}\le A\le B^{2-\delta}$,
\[
 \Xi_{6,S_0}(A,B)\gg_{\delta,\eps,S_0}
                   (AB)^{-\eps}A^{5/6}B^{1/3}.
\]
Consequently, for every $\eps>0$ and $Q\ge2$,
\[
                     \Xi_{6,S_0}(Q,Q)
                          \gg_{\eps,S_0}Q^{17/16-\eps}.
\]
\end{theorem}

The second conclusion is a full-cutoff statement, obtained
by taking columns shorter than $Q$.

\subsection{The two-Poisson estimate}

Retain the notation $F_{\mathfrak d,\mathbf e;U}$,
$D=\N\mathfrak d$, $X=L^2/q$, and $\Omega$ from
Section \ref{sec:gram}.

\begin{lemma}[The first-version Gram estimate]
\label{lem:classical-gram}
For a fixed nonnegative smooth profile $U$, $q\ge L\ge1$,
$X\ge1$, and $q/L$ sufficiently large relative to the
fixed data,
\begin{equation}\label{eq:classical-gram}
 \begin{split}
 \|F_{\mathfrak d,\mathbf e;U}(\,\cdot\,;L)\|_q^2
 ={}&C_{S,w}\II(U)^2q^{5/6}L^{-2/3}
                         \Omega(\mathfrak d,\mathbf e)\\
 &+O_\eps\left((2qLD)^\eps
          \{(q/L)(1+X^{1/12})+DX\}\right).
 \end{split}
\end{equation}
The constants are uniform in the moving exclusions,
exponents, and ray class. For $X\le1$, the upper bound
is $O_U(q/L)$.
\end{lemma}

\begin{proof}
Use the row-Poisson identity \eqref{eq:first-poisson}
and the local correlation \eqref{eq:local-correlation},
including its one-nonunit branches. Insert the same
coprime projector $n_i=\mathfrak b m_i$.
For this proof write $Q$ for the unrestricted row scale
and $X_Q=L^2/Q$. We assume $Q\ge L$.
The first frequency $k$ is nonzero; zero and nonzero
modes below refer to the subsequent column Poisson
summation.

The full two-column coefficient is bounded by $q_C$
and has a common period of norm
$O_S(Dq_Cq_k)$ in each column.
On the ball $q_Cq_k\ll X_Q$ and outer dyadic shells
$q_Cq_k\asymp\mathcal R X_Q$, write $\mathcal R=1$
for the ball. If $q_C\asymp C_0$ and
$q_{\mathfrak b}\asymp B_0$, the residual column
scale is $N=L/(C_0B_0)$.
Subtract the complete mean and apply Poisson in the
two columns. For fixed $\alpha>1$ and sufficiently
large $\beta$, the error for one tuple is
\begin{equation}\label{eq:classical-tuple-error}
 \ll \mathcal R^{-\beta}C_0N^2
       \min\left\{1,
             \left(\frac{D\mathcal R X_Q}{N}\right)^\alpha
           \right\}.
\end{equation}
For periods shorter than the column scale, this follows
from rapid decay at the nonzero dual lattice points.
For short residual ranges, bound the lattice sum and
its complete mean separately. This also handles an
empty subunit annulus. Rescaled derivatives of the
coupled smooth weight have arbitrary decay in
$\mathcal R$.

There are $O(C_0)$ common factors, $O(B_0)$ projector
ideals, and $O(\mathcal R X_Q/C_0)$ frequencies on a
nonempty shell. Including the factor $Q/L^3$ gives
\[
 \mathcal R^{-\beta}\frac QL
       \frac{\mathcal R X_Q}{C_0B_0}
       \min\left\{1,
          \left(\frac{D\mathcal R X_Q C_0B_0}{L}\right)^\alpha
           \right\}.
\]
Put $V=C_0B_0$ and $\Lambda=D\mathcal R X_Q/L$.
The dyadic sum of
$V^{-1}\min\{1,(\Lambda V)^\alpha\}$ is
$O(\Lambda)$: below $\Lambda V=1$ it grows towards
the endpoint, and above it the reciprocal scale
decays. The number of pairs of divisor dyads with
a given product costs a logarithm. Summing the
shells therefore gives
\begin{equation}\label{eq:classical-poisson-error}
 O_\eps\left((2QLD)^\eps\frac QL\frac{DX_Q^2}{L}\right)
                         =O_\eps((2QLD)^\eps DX_Q).
\end{equation}

The complete means are exactly those evaluated in
Section \ref{subsec:complete-means}. Their evaluation
does not use the plain moment: the exceptional
condition \eqref{eq:sixth-root}, the completion bound
\eqref{eq:mean-completion}, and the weighted ideal
count \eqref{eq:weighted-ideal-count} give the same
main term, with errors
$O_\eps((2QLD)^\eps(Q/L)(1+X_Q^{1/12}))$.
The diagonal also has size $O(Q/L)$.
This proves the unrestricted-row formula.

Restore $(m,S)=1$ by the finite row projector.
The refined ray classes make each $S$-supported
numerator twist constant on a column fiber.
Its main coefficient is thus
$C_{S,w}^{\,0}\prod_{p\in S}(1-q_p^{-5/6})=C_{S,w}$,
as in Section \ref{sec:gram}.
Taking $q/L$ sufficiently large places every
$Q=q/\N h$, $h\mid\prod_{p\in S}p$, in the range
just considered. The finite sum of the errors proves
\eqref{eq:classical-gram}.

If $X\le1$, the reduced additive fractions have
separation $\gg L^{-1}$ and total squared coefficient
mass $O_U(L^{-1})$. The packing bound gives a Gram
row sum $O(q+L^2)$, so Schur's inequality gives
$O_U((q+L^2)/L)=O_U(q/L)$. This argument is also
independent of the plain moment.
\end{proof}

The extra term in \eqref{eq:classical-gram} is
$DX=(q/L)DX^2/L$. At $(q,L,D)=(A,B,1)$ its ratio
to the main scale is
\[
 \frac{Y}{(A/B)Y^{1/6}}
       =\frac{Y^{11/6}}B=B^{8/3}A^{-11/6},
                         \qquad Y=B^2/A.
\]
This is the source of the threshold $A>B^{16/11}$.
Replacing $DX$ by $D$ requires the additional plain
moment used in the main text.

\subsection{The squarefree-column witness}

\begin{lemma}\label{lem:original-witness}
Fix $0<\delta<3/11$ and the smooth profiles $W,w$.
After a fixed enlargement $S\supseteq S_0$, for all
sufficiently large $B$ and
$B^{16/11+\delta}\le A\le B^{2-\delta}$, there is a
nonzero vector $\lambda$ supported on squarefree
$b\in I(S)$ with $\N b\le B$ such that
\[
 \|\lambda\|_2^2\ll B^{-1},\qquad
 \sum_{m\in I(S)}w(m/\sqrt A)|P_\lambda(m)|^2
                              \gg A^{5/6}B^{-2/3}.
\]
The row sum is over unrestricted ideals, evaluated at
their primary generators.
\end{lemma}

\begin{proof}
Put $Z=AB$, $Y=B^2/A$, and
$\mathcal M=A^{5/6}B^{-2/3}=(A/B)Y^{1/6}$.
Let $\mathcal A=F_{1,0;W}(\,\cdot\,;B)$ and let
$\mathcal A^{\rm sf}$ restrict its columns to
squarefree moduli. The $D=1$ case of
\eqref{eq:classical-gram} gives
\[
 \|\mathcal A\|_A^2
   =C_{S,w}\II(W)^2\mathcal M
      +O_\eps\left(Z^\eps
          \{(A/B)(1+Y^{1/12})+Y\}\right).
\]

Use the exact powerful-modulus decomposition and
prime-power components of Section \ref{sec:columns}.
For a component of norm $T=\N t$, its scales are
$q_t=AE/T$, $L_t=B/T$, and $X_t=Y/(TE)$.
The moving support is $D_t=\N\rad t$.
The ratio $q_t/L_t=(A/B)E$ tends to infinity
uniformly, so Lemma \ref{lem:classical-gram} applies
to every component with $X_t\ge1$; the short-column
bound handles the others.

The local main factors are \eqref{eq:local-main-factors}.
Their square roots have the same convergent local
sum $f(P)<2P^{-3/2}$. Hence a fixed enlargement of
$S$ makes the relative main contribution of all
nonunit powerful components less than $1/4$.
This uses only the explicit main term, not the
improved error of Proposition \ref{prop:gram}.

The squared component error from the classical
estimate is at most
\[
 Z^\eps E_{\rm bd}^{-1}
 \left[
   \frac AB\{T^{-1}+Y^{1/12}T^{-13/12}E^{-1/12}\}
                     +\frac{D_tY}{T^2E^2}
 \right].
\]
After square roots, use the first two elementary
divisor sums in \eqref{eq:three-sums} and
\[
 \sum_{t,\mathrm{components}}
                D_t^{1/2}T^{-1}E^{-1}E_{\rm bd}^{-1/2}
                                                    \ll1.
\]
For the last sum, the largest prime contribution is
$P^{-5/2}$, from exponent two, so the Euler product
converges. Minkowski's inequality gives
\[
 \|\mathcal A-\mathcal A^{\rm sf}\|_A
 \le \frac14\sqrt{C_{S,w}}\II(W)\sqrt{\mathcal M}
   +O_\eps\left(Z^\eps
       \{\sqrt{A/B}(1+Y^{1/24})+\sqrt Y\}\right).
\]
The error divided by $\sqrt{\mathcal M}$ is
\[
 O_\eps\left(Z^\eps
       \{Y^{-1/12}+Y^{-1/24}+Y^{11/12}/\sqrt B\}\right)
                                                     =o(1),
\]
because $Y\ge B^\delta$ and
$Y^{11/6}/B\le B^{-11\delta/6}$.
The all-modulus error is smaller than its main term
for the same reason. Thus
$\|\mathcal A^{\rm sf}\|_A^2\gg\mathcal M$.

For squarefree $b$, primitive Gauss evaluation gives
$G(b,-m)=G(b,1)\overline{\chi_b(-m)}$, including
the zero values. Write $m=\epsilon a$ with $a$
primary. One of the six unit classes carries at
least a sixth of the energy. Conjugating its
polynomial gives the character coefficients
\[
 \lambda_b=B^{-1}W(q_b/B)(q_b/B)^{-1/2}
          \overline{G(b,1)}\chi_b(-\epsilon)
          1_{b\ {\rm sf},\ (b,S)=1,\ [b]=\sigma_{\rm ray}}.
\]
Radiality preserves the weight, and ideal counting
gives $\|\lambda\|_2^2\ll B^{-1}$.
Its positive row energy makes the vector nonzero
and proves the lemma.
\end{proof}

\subsection{The original operator-norm transfer}

The transfer below may change the coefficient vector.
This is precisely the distinction from the
small-twist starting point in Section
\ref{sec:optimal-proof}, where the Gauss formula
must be retained for the later scale comparison.

\begin{proposition}\label{prop:original-transfer}
Fix $S,w$, $c>0$, $0<\delta<1/3$, and $\eps>0$.
Suppose $B^{4/3+\delta}\le A\le B^{2-\delta}$ and
a nonzero squarefree-column vector $\lambda$
satisfies
\[
 \sum_{m\in I(S)}w(m/\sqrt A)|P_\lambda(m)|^2
                  \ge cA^{5/6}B^{1/3}\|\lambda\|_2^2.
\]
Then
\[
 \Xi_{6,S}(A,B)\gg_{\delta,\eps,c,S}
                     (AB)^{-\eps}A^{5/6}B^{1/3}.
\]
\end{proposition}

\begin{proof}
Normalize $\|\lambda\|_2=1$ and write
$Z=AB$, $M=A^{5/6}B^{1/3}$, and $g=\delta/18$.
Let $\Psi_6(A,B)$ be the squared norm with
sixth-power-free rows and squarefree columns.
Choose
\[
       \eta=\min\{\eps/5,g/10\},\qquad
       \rho=\eta/4,\qquad T_6=Z^\eta.
\]
The sixth-power tail \eqref{eq:t6-tail} is $o(M)$.
The head has $O(T_6)$ ideals $t$; each is bounded by
$\Psi_6(A,B)$ after absorbing the mask
$1_{(b,t)=1}$ into the coefficients. The assumed
moment therefore gives $\Psi_6(A,B)\gg MZ^{-\eta}$.

For a sixth-power-free row, write
\[
                  u=a_1a_2^2a_3^3a_4^4a_5^5.
\]
The higher-factor head, defined by
$\prod_{j=2}^5\N a_j\le T$, has
$O(T\log^3(2T))$ tuples. Each contributes at most
$\Xi_{6,S}(A,B)$ for a unit-norm coefficient vector.
Combining this with \eqref{eq:higher-tail}, and
reserving a smaller preliminary loss for logarithms,
gives
\[
 \Psi_6(A,B)\ll Z^\rho
       \{T\Xi_{6,S}(A,B)+M(T^{-2/3}+Z^{-g})\}.
\]
With $T=Z^{3\eta}$ the error is
$O(M\{Z^{-7\eta/4}+Z^{\eta/4-g}\})=o(MZ^{-\eta})$.
Absorbing it yields
$\Xi_{6,S}(A,B)\gg MZ^{-17\eta/4}\gg MZ^{-\eps}$.
All tails used here were proved from the known sieve
upper bounds alone in Lemma \ref{lem:tails}.
\end{proof}

\begin{proof}[Proof of Theorem \ref{thm:unbalanced}]
Lemma \ref{lem:original-witness} gives a Rayleigh
quotient $\gg A^{5/6}B^{1/3}$ on unrestricted ideal
rows. Proposition \ref{prop:original-transfer} applies
with the same $\delta$, since $16/11>4/3$ and
$3/11<1/3$. The inclusion
$\Xi_{6,S_0}(A,B)\ge\Xi_{6,S}(A,B)$ returns to
the prescribed exclusions.

For the full cutoff put $\theta=11/16-\tau$ with
$0<\tau<3/16$ and $B_0=Q^\theta$.
Choose
\[
 0<\delta_0<
       \min\{3/11,\ 1/\theta-16/11,\ 2-1/\theta\}.
\]
The restricted-range result and column inclusion give
\[
 \Xi_{6,S_0}(Q,Q)\ge\Xi_{6,S_0}(Q,B_0)
             \gg Q^{17/16-\tau/3-(1+\theta)\eps_0}.
\]
Taking $\tau$ and then $\eps_0$ sufficiently small
proves the second assertion. Bounded parameter
ranges follow by decreasing the constant and using
the unit row and column.
\end{proof}

\section*{AI use disclosure}

The main framework of the present argument is due to the author.
Generative AI was used interactively in mathematical
exploration and the development and refinement of parts
of the argument. It also assisted with the organization
and language of the manuscript.

\end{document}